\documentclass{amsart}

\usepackage{amssymb}
\usepackage{amsthm}
\usepackage{amsmath}
\usepackage{enumitem}
\usepackage[mathscr]{euscript}

\usepackage{verbatim}

\usepackage{cancel}
\usepackage{soul}

\usepackage[normalem]{ulem}

\usepackage[foot]{amsaddr}

\usepackage{pgfplots}
\pgfplotsset{compat=1.16}

\usepackage{hyperref}

\theoremstyle{plain}
\newtheorem{proposition}{Proposition}[section]
\newtheorem{theorem}[proposition]{Theorem}
\newtheorem{lemma}[proposition]{Lemma}
\newtheorem{corollary}[proposition]{Corollary}
\theoremstyle{definition}
\newtheorem{example}[proposition]{Example}
\newtheorem{definition}[proposition]{Definition}
\newtheorem{observation}[proposition]{Observation}
\theoremstyle{remark}
\newtheorem{remark}[proposition]{Remark}

\DeclareMathOperator{\diam}{diam}

\DeclareMathOperator{\id}{id}

\DeclareMathOperator{\dist}{d}
\DeclareMathOperator{\Euc}{Euc}
\DeclareMathOperator{\Gr}{Gr}

\DeclareMathOperator{\Cc}{\mathcal{C}}

\DeclareMathOperator{\Oc}{\mathcal{O}}

\DeclareMathOperator{\Cb}{\mathbb{C}}
\DeclareMathOperator{\Db}{\mathbb{D}}

\DeclareMathOperator{\Kb}{\mathbb{K}}
\DeclareMathOperator{\Nb}{\mathbb{N}}

\DeclareMathOperator{\Rb}{\mathbb{R}}
\DeclareMathOperator{\Sb}{\mathbb{S}}

\DeclareMathOperator{\Zb}{\mathbb{Z}}

\newcommand{\abs}[1]{\left|#1\right|}
\newcommand{\norm}[1]{\left\|#1\right\|}

\newcommand{\ip}[1]{\left\langle #1\right\rangle}

\begin{document}

\title[Gromov hyperbolicity of intrinsic metrics]{Gromov hyperbolicity of intrinsic metrics from isoperimetric inequalities}

\author[Wang]{Tianqi Wang}
\email{tq.wang@yale.edu}
\address{Department of Mathematics, Yale University, USA}

\author[Zimmer]{Andrew Zimmer}
\email{amzimmer2@wisc.edu}
\address{Department of Mathematics, University of Wisconsin-Madison, USA}

\date{\today}

\subjclass[2020]{}

\begin{abstract} In this paper we investigate the Gromov hyperbolicity of the classical Kobayashi and Hilbert metrics, and the recently introduced minimal metric. Using the linear isoperimetric inequality characterization of Gromov hyperbolicity, we show if these metrics have an ``expanding property'' near the boundary, then they are Gromov hyperbolic. This provides a new characterization of the convex domains whose Hilbert metric is Gromov hyperbolic, a new proof of Balogh--Bonk's result that the Kobayashi metric is Gromov hyperbolic on a strongly pseudoconvex domain, a new proof of the second author's result that the Kobayashi metric is Gromov hyperbolic on a convex domain with finite type, and a new proof of Fiacchi's result that the minimal metric is Gromov hyperbolic on a strongly minimally convex domain. We also characterize the smoothly bounded convex domains where the  minimal metric is Gromov hyperbolic.

\end{abstract}

\maketitle

\setcounter{tocdepth}{1}
\tableofcontents

\section{Introduction}

In recent years there has been considerable interest in understanding when certain intrinsic metrics are Gromov hyperbolic. For instance, 
\begin{itemize}
\item Benoist~\cite{Benoist2003} characterized the convex domains for which the Hilbert metric is Gromov hyperbolic. 
\item Balogh--Bonk~\cite{BaloghBonk2000} proved that the Kobayashi metric is Gromov hyperbolic on a strongly pseudoconvex domain. The second author~\cite{ZimmerMathAnn2016} proved that the Kobayashi metric is Gromov hypberolic on a smoothly bounded convex domain if and only if the domain has finite type (see~\cite{ZimmerAdv2022} for a characterization in the non-smooth case). Fiacchi~\cite{Fiacchi2022} proved that the Kobayashi metric is Gromov hyperbolic on a finite type pseudoconvex domain in $\Cb^2$. Recently, Wang~\cite{Wang2024} and Li--Pu--Wang~\cite{LiPuWang2024} discovered new proofs that finite type implies Gromov hyperbolicity in the convex and dimension two case respectively.
\item Fiacchi~\cite{Fiacchi2023} showed that the minimal metric (a metric recently introduced by  Forstneri{\v c}--Kalaj~\cite{FK2021}) is Gromov hyperbolic on a  strongly minimally convex domain (an analogue of strongly pseudoconvex domains). 
\end{itemize}

In this paper we observe that the metrics in these previously considered cases share a common expansion property near the boundary (described in Theorem~\ref{thm:expanding metrics intro} below) and we show that this expansion property implies Gromov hyperbolicity via a linear isoperimetric inequality.

We are also able to use this approach to prove a new result: on a smoothly bounded convex domain, the minimal metric is Gromov hyperbolic if and only if every real affine 2-plane has finite order contact with the boundary. 

\subsection{Expanding metrics} 

In this paper we study distances on manifolds $M$ which are induced by a ``strongly integrable pseudo-metric'' (see Section~\ref{sec:psuedometrics} for the precise definition). These are families $\{ \norm{\cdot}_p\}_{p \in M}$ of functions $\norm{\cdot}_p : T_p M \rightarrow [0,+\infty)$ with certain regularity properties that induce a distance defined by 
$$
\dist(p,q) : = \inf \left\{ \int_a^b \norm{\sigma'(t)}_{\sigma(t)} dt : \begin{array}{cc} \sigma : [a,b] \rightarrow M \text{ absolutely continuous and } \\ \sigma(a) = p, \, \sigma(b) = q \end{array} \right\}.
$$
Using the linear  isoperimetric inequality characterization of Gromov hyperbolicity we will prove the following. 

\begin{theorem}[see Theorem~\ref{thm:expanding metrics}]\label{thm:expanding metrics intro}
Suppose $M$ is a compact manifold and $\norm{\cdot}$ is a strongly integrable pseudo-metric on $M \times (0,1]$ with the following properties:
\begin{enumerate}[label=(\alph*)]
\item\label{item:normal lines in intro} There exists $C_1 > 1$ such that
$$
\frac{1}{C_1t} \leq \norm{ \left( 0, \frac{d}{d t} \right) }_{(p,t)} \leq C_1 \frac{1}{t}
$$
for all $(p,t) \in M \times (0,1]$. 
\item\label{item:orthogonality in intro} There exists $C_2 >0$ such that 
$$
\norm{ \left( v, 0 \right) }_{(p,t)}+\norm{ \left( 0, \frac{d}{d t} \right) }_{(p,t)} \leq C_2 \norm{ \left( v, \frac{d}{d t} \right) }_{(p,t)}
$$
for all $(p,t) \in M \times (0,1]$ and $v \in T_p M$.
\item\label{item:expanding in intro} There exist $C_3, \lambda > 0$ such that 
$$
\norm{ \left( v,0 \right) }_{(p,s)} \geq C_3 \left(\frac{t}{s}\right)^{\lambda} \norm{ \left( v,0 \right) }_{(p,t)} 
$$
for all $p \in M$, $0 < s \leq t \leq 1$, and $v \in T_p M$. 
\end{enumerate}
Then the distance on $M \times (0,1]$ induced by $\norm{\cdot}$ is Gromov hyperbolic. 
\end{theorem} 

We will apply Theorem~\ref{thm:expanding metrics intro} as follows: suppose $\Omega$ is a smoothly bounded domain in $\Rb^d$ or $\Cb^d$. For $x \in \partial \Omega$, let $\mathbf{n}(x)$ denote the inward pointing unit normal vector at $x$. Then for $\delta > 0$ sufficiently small, the map $F : \partial \Omega \times (0, 1] \rightarrow \Omega$ defined by 
$$
F(x,t) = x + \delta t \mathbf{n}(x)
$$
is a diffeomorphism onto its image. Given an infinitesimal metric on $\Omega$, we can pull it back to an infinitesimal metric on $\partial \Omega \times (0,1]$. Since $\Omega \setminus F( \partial \Omega \times (0, 1])$ is relatively compact in $\Omega$, the induced distances will be quasi-isometric and so to verify that the distance on $\Omega$ is Gromov hyperbolic, it suffices to verify that the pull back metric satisfies the hypothesis of Theorem~\ref{thm:expanding metrics intro}. In this setting, 
\begin{itemize}
\item Condition~\ref{item:normal lines in intro} essentially says that normal lines can be parametrized to be quasi-geodesics, 
\item Condition~\ref{item:orthogonality in intro} says that the tangential and normal directions are uniformly transverse with respect to the metric, and 
\item Condition~\ref{item:expanding in intro}  says that the metric in the tangential direction is uniformly expanding as one approaches the boundary. 
\end{itemize} 

As a corollary we obtain new proofs of several prior results. 

\begin{corollary}[Balogh--Bonk~\cite{BaloghBonk2000}] The Kobayashi metric on a strongly pseudoconvex domain is Gromov hyperbolic. 
\end{corollary} 

Since the proof is short we include it here. 

\begin{proof} Suppose $\Omega \subset \Cb^d$ is a strongly pseudoconvex domain. As described above, for $\delta > 0$ sufficiently small, consider the map $F : \partial \Omega \times (0, 1] \rightarrow \Omega$ defined by 
$$
F(x,t) = x + \delta t \mathbf{n}(x).
$$
Classical estimates for the Kobayashi metric (for instance~\cite{MR1121150}) imply that the pull back of the Kobayashi metric under $F$ satisfies Theorem~\ref{thm:expanding metrics intro}. Hence the induced distance on $\partial \Omega \times (0, 1]$ is Gromov hyperbolic. Since $\Omega \setminus F(\partial \Omega \times (0,1])$ is relatively compact in $\Omega$, the Kobayashi distance on $\Omega$ and the distance on $\partial \Omega \times (0, 1]$ are quasi-isometric. Since Gromov hyperbolicity is a quasi-isometric invariant, the Kobayashi distance on $\Omega$ is Gromov hyperbolic. 
\end{proof} 

Using the estimates for the minimal metric in~\cite[Theorem 3.2]{Fiacchi2022} one can similarly deduce the following. 

\begin{corollary}[Fiacchi~\cite{Fiacchi2022}] The minimal metric on a strongly minimally convex domain is Gromov hyperbolic. 
\end{corollary} 

In ~\cite{BaloghBonk2000} and~\cite{Fiacchi2022}, the authors first establish a precise estimate for the (infinitesimal) metric and then use intricate arguments to estimate the associated distance to directly verify Gromov hyperbolicity. In contrast, using Theorem~\ref{thm:expanding metrics intro} we only require estimates for  the (infinitesimal) metrics to deduce Gromov hyperbolicity. 

For finite type pseudoconvex  domains in $\Cb^2$, it is possible to use Catlin's estimates~\cite{Catlin1989} to show that the Kobayashi metric satisfies the hypothesis of Theorem~\ref{thm:expanding metrics intro} on a neighborhood of the boundary. Thus Theorem~\ref{thm:expanding metrics intro} can also be used to provide a new proof of Fiacchi's~\cite{Fiacchi2022} result that the Kobasyashi metric is Gromov hyperbolic on a finite type pseudoconvex domain in $\Cb^2$. 

\subsection{Generalized quasi-hyperbolic metrics}  Next we apply Theorem~\ref{thm:expanding metrics intro} to study the Gromov hyperbolicity of the Hilbert, Kobayashi, and minimal metric on bounded convex domains. We will study all of these metrics simultaneously  by introducing a new family of metrics.

Suppose $\Kb$ is either the real numbers $\Rb$ or the complex numbers $\Cb$. Let $\Omega \subset \Kb^d$ be a bounded domain. Given $p \in \Omega$ let 
$$
\delta_\Omega(p): = \dist_{\rm Euc}(p,\partial \Omega)
$$
denote the Euclidean distance between $p$ and the boundary $\partial \Omega$. Then the well-studied  \emph{quasi-hyperbolic metric on $\Omega$} is defined by 
$$
\mathfrak{q}_\Omega(p;v) = \frac{\norm{v}}{\delta_\Omega(p)}
$$
where $p \in \Omega$ and $v \in \Kb^d \simeq T_p \Omega$. 

We define variants of the quasi-hyperbolic metric as follows. First given $p \in \Omega$ and non-zero $v \in \Kb^d \simeq T_p \Omega$, the distance to the boundary from $p$ in the direction of $v$ is 
$$
\delta_\Omega(p;v) := \dist_{\rm Euc}(p, (p + \Kb\cdot v) \cap \partial \Omega). 
$$
More generally,  for $k=1,\dots, d$ we define 
\begin{align*}
\delta^{(k)}_\Omega(p;v) = \sup\left\{ \dist_{\rm Euc}(p, (p+V) \cap \partial \Omega)  : \begin{array}{c}
    V \subset \Kb^d \text{ a $\Kb$-linear subspace with $v \in V$} \\
    \text{and } {\rm dim}_{\Kb} V = k
  \end{array} \right\}.
\end{align*}

Then the \emph{generalized $k$-quasi-hyperbolic metric on $\Omega$} is defined by 
$$
\mathfrak{q}^{(k)}_\Omega(p;v) = \frac{\norm{v}}{\delta_\Omega^{(k)}(p;v)}
$$
where $p \in \Omega$ and $v \in \Kb^d \simeq T_p \Omega$. Let $\dist_\Omega^{(k)}$ denote the distance on $\Omega$ induced by $\mathfrak{q}^{(k)}_\Omega$, i.e. 
$$
\dist_\Omega^{(k)}(p,q) = \inf \left\{ \int_a^b \mathfrak{q}^{(k)}_\Omega(\sigma(t); \sigma'(t))dt: \begin{array}{c} \sigma : [a,b] \rightarrow \Omega \text{ absolutely continuous } \\ \text{and } \sigma(a) = p, \, \sigma(b) = q \end{array} \right\}.
$$ 

The definition of this metric may seem somewhat unnatural, but it encompasses all of the metrics we are interested in. 

\begin{example}\label{example: special quasi hyperbolic metric} \,
\begin{enumerate}
\item When $k=d$, the metric $\mathfrak{q}^{(d)}_\Omega$ coincides with the quasi-hyperbolic metric $\mathfrak{q}_\Omega$.
\item When $\Omega$ is convex, $\Kb = \Rb$, and $k=1$, the metric $\mathfrak{q}^{(1)}_\Omega$ is biLipschitz to the Hilbert metric (this follows from the definition, see for instance~\cite[Section 3.2.1]{Benoist2004}). 
\item When $\Omega$ is convex, $\Kb = \Cb$, and $k=1$, the metric  $\mathfrak{q}^{(1)}_\Omega$ is biLipschitz to the Kobayashi metric (see for instance~\cite[Theorem 4.1]{BP1994}, \cite[Theorem 5]{Graham1991}, or \cite[Theorem 2.2]{Frankel1991}).
\item When $\Omega$ is convex, $\Kb = \Rb$, $d\geq 3$, and $k=2$, the metric  $\mathfrak{q}^{(2)}_\Omega$ is biLipschitz to the minimal metric. This follows from results in the literature~\cite{Fiacchi2023,DDF2021}, see Proposition~\ref{prop:minimal_metruc_is_Lip_to_affine} below. 
\end{enumerate} 
\end{example}

We will establish the following characterization of Gromov hyperbolicity for the generalized quasi-hyperbolic metrics.

\begin{theorem}[see Theorems~\ref{thm:2 implies 1 in main theorem} and~\ref{thm:1 implies 2 in main theorem}]\label{thm:characterization in nonsmooth case} Suppose $\Omega \subset \Kb^d$ is a bounded convex domain. Fix $p \in \Omega$. Then for $x \in \partial\Omega$ and $t \in [0,1]$, let 
$$
x_t : = (1-t)x+tp.
$$
Then the following are equivalent: 
\begin{enumerate}
\item $(\Omega, \dist_\Omega^{(k)})$ is Gromov hyperbolic.
\item There exist $C, \lambda > 0$ such that: if $0 <  s < t \leq 1$, $x \in \partial \Omega$, and $V \in \Gr_k(\Kb^d)$, then 
$$
 \dist_{\rm Euc}(x_s, (x_s+V) \cap \partial \Omega) \leq C \left( \frac{s}{t} \right)^{\lambda}  \dist_{\rm Euc}(x_t, (x_t+V) \cap \partial \Omega).
$$
\end{enumerate} 
\end{theorem}

We prove (2) $\Rightarrow$ (1) using a generalization of Theorem~\ref{thm:expanding metrics intro} (see Section~\ref{sec:expanding metrics}). The converse direction involves constructing non-slim geodesic triangles and uses ideas from~\cite{ZimmerMathAnn2016,ZimmerAdv2022}.

Recall that when $\Omega$ is convex and $\Kb = \Rb$, the metric $\mathfrak{q}^{(1)}_\Omega$ is biLipschitz to the Hilbert metric. So Theorem~\ref{thm:characterization in nonsmooth case} provides a characterization of the bounded convex domains in $\Rb^d$ where the Hilbert metric is Gromov hyperbolic. Previously, Benoist established a different characterization~\cite{Benoist2003}, but it does not seem easy to show that the two characterizations are equivalent. So in this special case Theorem~\ref{thm:characterization in nonsmooth case}  should be viewed as complementary to Benoist's theorem instead of providing a new proof. It would be very interesting to know if a version of Benoist's characterization extends to generalized quasi-hyperbolic metrics. 

Recall that when $\Omega$ is convex and $\Kb = \Cb$, the metric $\mathfrak{q}^{(1)}_\Omega$ is biLipschitz to the Kobayashi metric. In this case, a characterization of Gromov hyperbolicity was given in~\cite{ZimmerAdv2022} and a rescaling argument (as in~\cite[Section 6]{ZimmerAdv2022}) can be used to show that this characterization is equivalent to the one given in Theorem~\ref{thm:characterization in nonsmooth case}.

When the convex domain is smoothly bounded, condition (2) in Theorem~\ref{thm:characterization in nonsmooth case} can be reformulated in terms of affine planes having finite order of contact with the boundary (see Definition~\ref{defn:finite order contact} for a precise definition). 

\begin{theorem}[see Theorem~\ref{thm:smooth case}]\label{thm:characterization in smooth case} Suppose $\Omega \subset \Kb^d$ is a smoothly bounded convex domain. Then the generalized $k$-quasi-hyperbolic metric on $\Omega$ is Gromov hyperbolic if and only if every $\Kb$-affine $k$-plane has finite order contact with $\partial \Omega$. 
\end{theorem}

For a smoothly bounded convex domain in $\Cb^d$, every complex line having finite order contact with the boundary is equivalent to having finite type in the sense of D'Angelo~\cite{McNeal92,BS1992} and so Theorem~\ref{thm:characterization in smooth case} has the following corollary.

\begin{corollary}[Z. \cite{ZimmerMathAnn2016}] Suppose $\Omega \subset \Cb^d$ is a smoothly bounded convex domain. Then the Kobayashi metric on $\Omega$ is Gromov hyperbolic if and only if $\partial\Omega$ has finite type in the sense of D'Angelo. 
\end{corollary} 

Theorem~\ref{thm:characterization in smooth case} also provides a characterization of the smoothly bounded convex domains where the minimal metric is Gromov hyperbolic. 

\begin{corollary}[see Corollary~\ref{cor:characterization in smooth case for minimal metric in paper}]\label{thm:characterization in smooth case for minimal metric} Suppose $d \geq 3$ and $\Omega \subset \Rb^d$ is a smoothly bounded convex domain. Then the minimal metric on $\Omega$ is Gromov hyperbolic if and only if every real affine $2$-plane has finite order contact with $\partial \Omega$. 
\end{corollary}

\subsection{Outline of paper} Section~\ref{sec:prelim} is expository: we fix any possibly ambiguous notation and recall some important definitions.  In Section~\ref{sec:expanding metrics} we prove a more technical version of Theorem~\ref{thm:expanding metrics intro} (the technicality being necessary to handle non-smooth domains in Theorem~\ref{thm:characterization in nonsmooth case}).

In Sections ~\ref{sec:distance to the boundary}, ~\ref{sec:QH metric almost orthogonal directions},  and ~\ref{sec:a sufficient condition for GH} we apply the more technical version of Theorem~\ref{thm:expanding metrics intro} to prove the implication $(2) \Rightarrow (1)$ in  Theorem~\ref{thm:characterization in nonsmooth case}. The most difficult part is in Sections ~\ref{sec:distance to the boundary} and ~\ref{sec:QH metric almost orthogonal directions} where we establish estimates on the generalized quasi-hyperbolic metrics to verify Condition~\ref{item:orthogonality in intro} in Theorem~\ref{thm:expanding metrics intro}.

In Sections~\ref{sec:distance estimates for QH} and ~\ref{sec:a necessary condition for GH} we prove the implication $(1) \Rightarrow (2)$ in  Theorem~\ref{thm:characterization in nonsmooth case}. In Section~\ref{sec:distance estimates for QH} we show that certain Euclidean line segments can be parametrized to be quasi-geodesics. In Section~\ref{sec:a necessary condition for GH} we use these quasi-geodesics to prove $(1) \Rightarrow (2)$.

In Section~\ref{sec:the smooth case} we prove Theorem~\ref{thm:characterization in smooth case} and in Section~\ref{sec:minimal metric} we use results in the literature to show that the minimal metric is biLipschitz to a generalized quasi-hyperbolic metric. 

\subsection*{Acknowledgements} Before joining Yale University, Wang was partially supported by the NUS-MOE grant A-8000-458-00-00 and by the Merlion PhD program 2021. Zimmer was partially supported by a Sloan research fellowship and grant DMS-2105580  from the National Science Foundation.

Wang also thanks University of Wisconsin-Madison for hospitality during a visit where the work on this project started.

\section{Preliminaries}\label{sec:prelim}

\subsection{Possibly ambiguous notations} We fix some possibly ambiguous notations:
\begin{enumerate} 
\item A domain is a connected open set. 
\item $\norm{\cdot}$ is the $\ell^2$-norm on $\Kb^d$, $\dist_{\rm Euc}(x,y) = \norm{x-y}$ is the associated distance, and $B_{\Kb^d}(p,r)$ is the Euclidean ball of radius $r$ centered at $p$. 
\item $\Gr_k(\Kb^d)$ is the Grassmanian of linear subspaces in $\Kb^d$ with $\Kb$-dimension $k$. 

\end{enumerate} 

\subsection{Geodesic, length, and proper metric spaces} Suppose $(X,\dist)$ is a metric space. A \emph{(unit speed) geodesic} is a map $\sigma : I \rightarrow X$ of an interval $I \subset \Rb$ where $\dist(\sigma(s), \sigma(t)) = \abs{t-s}$ for all $s,t \in I$. A metric space is \emph{geodesic} if every two points are joined by a geodesic and \emph{proper} if every closed and bounded set is compact. 

The \emph{length} of a curve $\sigma : [a,b] \rightarrow X$ is 
$$
L(\sigma): = \sup \left\{ \sum_{i=1}^{n-1} \dist(\sigma(t_i), \sigma(t_{i+1})) : a = t_1 < \dots < t_n = b\right\}
$$
and $\sigma$ is \emph{rectifiable} if $L(\sigma) < +\infty$. The metric $\dist$ is a \emph{length metric} if 
$$
\dist(p,q) = \inf\{ L(\sigma) :  \sigma :[a,b] \rightarrow X \text{ is rectifiable and } \sigma(a) = p, \, \sigma(b) = q\}
$$
for all $p,q \in X$. 

Clearly a geodesic space is a length space and for proper spaces, the converse is also true (see for instance Proposition 3.7 and Corollary 3.9 in Chapter I.1 of~\cite{BH1999}). 

\begin{proposition}\label{prop:HopfRinow} If $(X,\dist)$ is a proper length metric space, then $(X,\dist)$ is geodesic. \end{proposition}

\subsection{Gromov hyperbolic spaces} In this subsection we recall some basic properties of Gromov hyperbolic metric spaces, for more details see~\cite{BH1999} or~\cite{MR3558533}.

A \emph{geodesic triangle} in a metric space $(X,\dist)$ is a choice of three points in $X$ and a choice of geodesic segments joining these points. A geodesic triangle is said to be \emph{$\delta$-slim} if any point on any of the sides of the triangle is within distance $\delta$ of the other two sides.

\begin{definition} 
A proper geodesic metric space is \emph{$\delta$-hyperbolic} if every geodesic triangle is $\delta$-slim, and it is \emph{Gromov hyperbolic} if it is $\delta$-hyperbolic for some $\delta \geq 0$.
\end{definition} 

Given $A \geq 1$, $B \geq 0$, a \emph{$(A,B)$-quasi-geodesic} in a metric space $(X,\dist)$ is a map $\sigma : I \rightarrow X$ of an interval $I \subset \Rb$ such that 
$$
\frac{1}{A} \abs{t-s} - B \leq \dist(\sigma(s), \sigma(t)) \leq  A\abs{t-s}+B
$$
for all $s,t \in I$. 

Recall that the Hausdorff pseudo-distance between two subsets $A,B$ in a metric space is
$$
\dist^{\rm Haus}(A,B) = \max\left\{ \sup_{a \in A} \dist(a,B), \sup_{b \in B} \dist(b,A)\right\}.
$$
Gromov hyperbolic metric spaces have the following fellow traveling property for quasi-geodesics (for a proof see for instance Theorem 1.7 in Chapter III.H of~\cite{BH1999}).

\begin{theorem}\label{thm:shadowing property} Suppose $(X,\dist)$ is Gromov hyperbolic. For any $A \geq 1$ and $B \geq 0$ there exists $D = D(A, B)> 0$ such that: If $\sigma_1 : [a,b] \rightarrow X$  and $\sigma_2 : [a',b'] \rightarrow X$ are $(A,B)$-quasi-geodesics with the same endpoints, then 
$$
\dist^{\rm Haus}\big(\sigma_1([a,b]), \sigma_2([a',b'])\big) \leq D.
$$
\end{theorem} 

Next we recall the linear isoperimetric inequality characterization of Gromov hyperbolicity. Let $(X,\dist)$ be a proper geodesic metric space. Given a closed curve $\sigma :\Sb^1 \to X$ and  $R>0$, a pair $(\Phi,P)$ is an \emph{$R$-filling} of $\sigma$ if $P$ is a triangulation of the closed disc $\overline{\Db}$ and $\Phi:\overline{\Db} \to X$ is a (not necessarily continuous) map where $\Phi|_{\Sb^1} = \sigma$ and the $\Phi$-image of any triangle in $P$ has diameter at most $R$. Then, the \emph{$R$-area of $\sigma$} is 
\[
{\rm Area}_R(\sigma) := \min \{\text{number of triangles in }P: (\Phi,P)\text{ an }R \text{-filling of }\sigma \}.
\]

\begin{theorem}[{see for instance \cite[Chapter H, Theorem 2.9]{BH1999}}]\label{thm:isomperimetric char of GH} Assume $(X,\dist)$ is a proper geodesic metric space. If there exist $R,A>0$ and $B\geq 0$ such that 
\[{\rm Area}_R(\sigma)\leq A L(\sigma) + B\]
for any rectifiable closed curve $\sigma:\Sb^1 \to X$,  then $(X,\dist)$ is Gromov hyperbolic.
\end{theorem}

\subsection{Pseudometrics}\label{sec:psuedometrics}

Suppose  $M$ is a connected $n$-manifold. A \emph{pseudo-metric} on $M$ is a family $\{\norm{\cdot}_{p} \}_{p \in M}$ where 
\begin{enumerate} 
\item $\norm{\cdot}_p : T_p M \rightarrow [0,+\infty)$;
\item\label{item:scaling} $\norm{tv}_p = \abs{t} \norm{v}_p$ for all $t \in \Rb$ and $v \in T_pM$;
\item\label{item:locally Euclidean} for every coordinate chart $(U,\psi)$ and $p \in U$, there exist $C > 1$ and an open neighborhood $V$ of $p$ in $U$  such that 
$$
\frac{1}{C}\norm{d(\psi)_q v} \leq \norm{v}_q \leq C \norm{d(\psi)_q v}
$$
for all $q \in V$ and $v \in T_q M$. 
\end{enumerate} 

If $I \subset \Rb$ is an interval, a curve $\sigma : I \rightarrow M$ is \emph{(locally) absolutely continuous} if for every $t_0 \in I$ and every coordinate chart $(U,\psi)$ containing $\sigma(t_0)$ there exists $\epsilon > 0$ such that 
$$
\psi \circ \sigma|_{I \cap [t_0-\epsilon, t_0+\epsilon]} : I \cap [t_0-\epsilon, t_0+\epsilon] \rightarrow \Rb^n
$$
is absolutely continuous. Notice that for such a curve, the derivative $\sigma'(t) \in T_{\sigma(t)} M$ exists for almost-every $t \in I$ with respect to the Lebesgue measure. 

Following~\cite{Venturini1989}, we say that a pseudo-metric $\{\norm{\cdot}_{p} \}_{p \in M}$ is \emph{strongly integrable} if for every absolutely continuous curve $\sigma : [a,b] \rightarrow M$ the almost everywhere defined function 
$$
t \in [a,b] \mapsto \norm{\sigma'(t)}_{\sigma(t)} \in [0,+\infty)
$$
is Lebesgue integrable. Given such a curve, its \emph{integrated length} is 
$$
L_I(\sigma) : = \int_a^b \norm{\sigma'(t)}_{\sigma(t)} dt.
$$
The \emph{integrated distance} $\dist : M \times M \rightarrow [0,\infty)$ associated to $\{\norm{\cdot}_{p} \}_{p \in M}$ is  defined  by 
$$
\dist(p,q) : = \inf \left\{ L_I(\sigma) : \sigma : [a,b] \rightarrow M \text{ abs. cont. and } \sigma(a) = p, \, \sigma(b) = q\right\}.
$$
Then $\dist$ is a metric on $M$ which generates the standard topology:
\begin{enumerate}[label=(\alph*)]
\item By construction, $\dist$ satisfies the triangle inequality. 
\item Property~\eqref{item:scaling} implies that $\dist(p,q) = \dist(q,p)$ for all $p,q \in M$. 
\item Property~\eqref{item:locally Euclidean} implies that $\dist(p,q) > 0$ when $p \neq q$ and that $\dist$ generates the standard topology. 
\end{enumerate} 
Notice that, by definition, $(M,\dist)$ is also a length metric space. 

Finally, we show that the generalized quasi-hyperbolic metrics are strongly integrable. 

\begin{proposition}\label{prop:q is strongly integrable} If $\Omega \subset \Kb^d$ is a bounded domain, then $\mathfrak{q}_\Omega^{(k)}$ is a strongly integrable pseudo-metric. 
\end{proposition} 

\begin{proof} Property~\eqref{item:scaling} follows from the definition. To verify Property~\eqref{item:locally Euclidean}, fix $p \in \Omega$. Then fix $r,R > 0$ such that $B_{\Kb^d}(p,2r) \subset \Omega \subset B_{\Kb^d}(p,R)$. Then 
$$
\frac{\norm{v}}{R} \leq \mathfrak{q}_\Omega^{(k)}(q;v) \leq \frac{\norm{v}}{r}
$$
for all $q \in B_{\Kb^d}(p,r)$ and $v \in \Kb^d$. So Property~\eqref{item:locally Euclidean} holds on the chart $(B_{\Kb^d}(p,r), \id)$. Hence $\mathfrak{q}_\Omega^{(k)}$ is a pseudo-metric. Since $\Omega$ is open, the function 
$$
(p,v) \mapsto \mathfrak{q}_\Omega^{(k)}(p;v)
$$
is upper semicontinuous. So $\mathfrak{q}_\Omega^{(k)}$ is a strongly integrable.

\end{proof}


\section{Expanding metrics}\label{sec:expanding metrics}


In this section we prove a more technical version of Theorem~\ref{thm:expanding metrics intro} from the introduction. This added technicality is needed to handle domains with non-smooth boundary in Theorem~\ref{thm:characterization in nonsmooth case}.

Suppose $\Omega$ is a manifold, $M$ is a closed manifold, and $\Phi : M \times (0,1] \rightarrow \Omega$ is a map with the following properties: 
\begin{enumerate} 
\item $\Phi$ is a homeomorphism onto its image,
\item $\Phi$ is locally biLipschitz (i.e. locally biLipschitz in any choice of local coordinates), and 
\item $\Omega \setminus \Phi(M \times (0,1])$ is relatively compact in $\Omega$. 
\end{enumerate}

\begin{theorem}\label{thm:expanding metrics}  Suppose $\{ \norm{\cdot}_z \}_{z \in \Omega}$ is a strongly integrable pseudo-metric on $\Omega$ and $\dist$ is the associated integrated distance. Assume:
\begin{enumerate}[label=(\alph*)]
\item\label{item:vertical estimate} There exists $C_1 > 1$ such that: if $\sigma : [a,b] \rightarrow (0,1]$ is absolutely continuous and $p \in M$, then
$$
\frac{1}{C_1}\frac{\abs{\sigma'(t)}}{\sigma(t)} \leq \norm{ \frac{d}{dt} \Phi(p,\sigma(t)) }_{\Phi(p,\sigma(t))} \leq C_1\frac{\abs{\sigma'(t)}}{\sigma(t)}
$$
for almost every $t \in [a,b]$. 
\item\label{item:almost_orthogonal_splitting} There exists $C_2 > 1$ such that: if $\sigma = (\sigma^1, \sigma^2) : [a,b] \rightarrow M \times (0,1]$ is absolutely continuous, then 
\begin{align*}
 \norm{ \left. \frac{d}{ds}\right|_{s=t} \Phi(\sigma^1(s),\sigma^2(t)) }_{\Phi(\sigma(t))} + \norm{ \left. \frac{d}{ds}\right|_{s=t} \Phi(\sigma^1(t),\sigma^2(s)) }_{\Phi(\sigma(t))} \leq C_2 \norm{ \frac{d}{dt} \Phi(\sigma(t)) }_{\Phi(\sigma(t))}
\end{align*}
for almost every $t \in [a,b]$. 
\item\label{item:expansion near boundary} There exist $C_3, \lambda > 0$ such that: if $\sigma : [a,b] \rightarrow M$ is absolutely continuous and $0 < s_1 \leq s_2 \leq 1$, then 
$$
\norm{ \frac{d}{dt} \Phi(\sigma(t),s_1) }_{\Phi(\sigma(t),s_1)} \geq C_3 \left(\frac{s_2}{s_1}\right)^{\lambda} \norm{ \frac{d}{dt}\Phi(\sigma(t),s_2) }_{\Phi(\sigma(t),s_2)} 
$$
for almost every $t \in [a,b]$. 
\end{enumerate} 
Then $(\Omega, \dist)$ is Gromov hyperbolic. 
\end{theorem} 

\begin{remark} Notice that Theorem~\ref{thm:expanding metrics intro} is Theorem~\ref{thm:expanding metrics} in the special case when $\Omega = M \times (0,1]$ and $\Phi = \id$. \end{remark} 

The rest of the section is devoted to a proof of the theorem. We first define a distance on $M \times (0,1]$. Given an absolutely continuous curve $\sigma : [a,b] \rightarrow M \times (0,1]$, we define 
$$
L_{I}(\sigma): =  \int_a^b \norm{(\Phi \circ\sigma)'(t)}_{(\Phi \circ \sigma)(t)} dt.
$$
Since $\Phi$ is locally biLipschitz, a curve  $\sigma : [a,b] \rightarrow M \times (0,1]$ is absolutely continuous if and only if $\Phi \circ \sigma : [a,b] \rightarrow \Omega$ is absolutely continuous. Hence $L_I(\sigma)$ is well defined. 

Then we define a distance on $M \times [0,1)$ by 
$$
\dist_1(p,q) : = \inf \left\{ L_{I}(\sigma) : \begin{array}{c} \sigma : [a,b] \rightarrow M \times (0,1] \text{ abs. cont. and }\\  \sigma(a) = p, \, \sigma(b) = q \end{array} \right\}.
$$
Since $\Phi$ is locally biLipschitz and $\norm{\cdot}$ is a strongly integrable pseudo-metric, $\dist_1$ is a length metric on $M \times [0,1)$ which generates the standard topology. 

Since $\Omega \setminus \Phi(M \times (0,1])$ is relatively compact in $\Omega$, we have the following. 

\begin{observation} The metric spaces $(\Omega, \dist)$ and $(M \times (0,1], \dist_1)$ are quasi-isometric, i.e. there exist $A > 1$, $B > 0$, and $R > 0$ such that 
$$
\frac{1}{A} \dist_1(p,q) - B \leq \dist(\Phi(p), \Phi(q)) \leq A \dist_1(p,q) + B
$$
for all $p,q \in M \times (0,1]$ and 
$$
\dist(x, {\rm image}(\Phi)) \leq R
$$
for all $x \in \Omega$.
\end{observation} 

Since Gromov hyperbolicity is preserved by quasi-isometries (see Theorem 1.9 in Chapter III.H in~\cite{BH1999}), it suffices to show that $(M \times (0,1], \dist_1)$ is Gromov hyperbolic. 

We first establish bounds on $\dist_1$ and verify that $\dist_1$ is proper.

\begin{lemma}\label{lem:vertical estimates} If $x,y \in M$ and $s,t \in (0,1]$, then 
\begin{equation*}
\frac{1}{C_1C_2} \abs{\log \frac{s}{t}} \leq \dist_1\Big( (x,s), (y,t)\Big)
\end{equation*}
and 
$$
\dist_1\Big( (x,s), (x,t)\Big) \leq C_1\abs{\log \frac{s}{t}} .
$$

\end{lemma} 

\begin{proof} We can assume that $t \geq s$. Consider an absolutely continuous curve $\sigma=(\sigma^1,\sigma^2) : [a,b] \rightarrow M \times (0,1]$ joining $(x,s)$ to $(y,t)$. By ~\ref{item:almost_orthogonal_splitting} and~\ref{item:vertical estimate} we have 
\begin{align*}
L_I(\sigma) & \geq \frac{1}{C_2} \int_a^b \norm{\left.\frac{d}{dr}\right|_{r=t} \Phi(\sigma^1(t),\sigma^2(r)) }_{\Phi(\sigma(t))} dt \geq  \frac{1}{C_1C_2} \int_a^b \frac{\abs{\sigma^{2\prime}(t)}}{\sigma^2(t)} dt \\
& \geq \frac{1}{C_1C_2} \log \frac{\sigma^2(b)}{\sigma^2(a)} = \frac{1}{C_1C_2} \log \frac{t}{s} = \frac{1}{C_1C_2} \abs{\log \frac{s}{t}}.
\end{align*}
Since $\sigma$ was an arbitrary absolutely continuous curve joining $(x,s)$ to $(y,t)$, 
$$
\frac{1}{C_1C_2} \abs{\log \frac{s}{t}} \leq \dist_1\Big( (x,s), (y,t)\Big).
$$

The upper bound follows from considering the curve $r \mapsto (x, e^{-r} t)$ and using~\ref{item:vertical estimate}. 
\end{proof} 

Since $M$ is compact, Lemma~\ref{lem:vertical estimates} implies that $\dist_1$ is proper, and hence geodesic by Proposition~\ref{prop:HopfRinow}. We will verify that $\dist_1$ is Gromov hyperbolic using the linear isoperimetric inequality characterization of Gromov hyperbolicity stated in Theorem~\ref{thm:isomperimetric char of GH}. This argument  has three main steps: 
\begin{enumerate} 
\item First, we show that any rectifiable closed curve can be approximated by an absolutely continuous closed curve (Lemma~\ref{lem:rectifiable to abs cont}). 
\item Second, we show that any absolutely continuous closed curve can be approximated by a curve which consists of ``horizontal'' and ``vertical'' pieces (Lemma~\ref{lem:abs cont to vert/hori}). 
\item Finally, we verify the linear isoperimetric inequality for curves consisting of ``horizontal'' and ``vertical'' pieces (Lemma~\ref{lem:area bound for vert/hori curves}). 
\end{enumerate} 
The argument in Step 3 is  inspired by the proof of~\cite[Proposition 3.7]{MR2448064}.

We begin by fixing some constants. First fix $T_0 > 0$ such that 
$$
C_3 e^{\lambda T_0} >2.
$$
Then let 
$$
L := \max\left\{C_1 T_0,\frac{C_2}{C_3}\right\}
$$
and 
 $$
 R := \max\left\{ 2, \, 2C_1(C_1C_2 + T_0), \, \diam\left(M \times [e^{-T_0}, 1]\right), \, \frac{9L}{2} \right\}. 
 $$

\begin{lemma}\label{lem:rectifiable to abs cont} Given any rectifiable closed curve $\sigma : [a,b] \rightarrow M \times (0,1]$ there exists an absolutely continuous closed curve $\tilde\sigma : [a.b] \rightarrow M \times (0,1]$ such that 
$$
L_I(\tilde \sigma) \leq L(\sigma) + 1
$$
and 
$$
{\rm Area}_R(\tilde \sigma) \leq {\rm Area}_R(\sigma) + L(\sigma)+1.
$$
\end{lemma} 

\begin{proof} Let $N : = \lfloor L(\sigma) \rfloor+1$. Then we can find a partition $a = t_1 < \dots < t_{N+1} = b$ such that 
 $$
L\left(\sigma|_{[t_i,t_{i+1}]}\right) < 1
 $$
 for all $i=1,\dots, N$. Then 
 $$
  \dist_1(\sigma(t_i), \sigma(t_{i+1})) \leq L\left(\sigma|_{[t_i,t_{i+1}]}\right) < 1.
  $$
 So by the definition of $\dist_1$, for each $i=1,\dots, N$ there exists an absolutely continuous closed curve $\tilde\sigma_i : [t_i,t_{i+1}] \rightarrow M \times (0,1]$ such that $\tilde\sigma_i(t_i) = \sigma(t_i)$, $\tilde\sigma_i(t_{i+1}) = \sigma(t_{i+1})$, and 
 $$
 L_I(\tilde \sigma_i) <1.
 $$
 So if $\tilde \sigma$ is the concatenation of $\tilde\sigma_1,\dots, \tilde \sigma_N$, then 
  $$
 L_I(\tilde \sigma) <N \leq L(\sigma)+1.
 $$
Further, 
  $$
 {\rm diam}\left(\sigma|_{[t_i,t_{i+1}]} \cup \tilde \sigma_i \right) < 2 \leq R
 $$
 and so 
 \begin{equation*}
{\rm Area}_R(\tilde \sigma) \leq {\rm Area}_R(\sigma) + N \leq {\rm Area}_R(\sigma)+L(\sigma)+1. \qedhere
\end{equation*}

\end{proof}

Next we show that an absolutely continuous closed curve can be approximated by a curve which consists of ``horizontal'' and ``vertical'' pieces. More precisely, we say that a curve $\sigma=(\sigma^1,\sigma^2) :[a,b] \rightarrow M \times (0,1]$ is 
\begin{itemize}
\item \emph{vertical} if $\sigma^1$ is constant and $\sigma^2$ induces a diffeomorphism 
$$[a,b] \rightarrow [e^{-kT_0}, e^{-(k+1)T_0}]$$
 for some $k \in \Zb_{\geq 0}$, 
\item \emph{horizontal} if $\sigma^2$ is constant and equals $e^{-kT_0}$ for some $k \in \Zb_{\geq 0}$ and $L_I(\sigma) \leq L$.
\end{itemize}

Notice that if $\sigma$ is vertical, then Condition~\ref{item:vertical estimate}  implies that 
\begin{equation}\label{eqn:vertical segment length}
\frac{1}{C_1} T_0 \leq L_I(\sigma) \leq C_1T_0 \leq L. 
\end{equation} 
A curve $\sigma:[a,b] \rightarrow M \times (0,1]$ has \emph{property-$(\star)$} if there exists a partition $a=t_0 < \cdots < t_N = b$ where each $\sigma|_{[t_i,t_i+1]}$ is either vertical or horizontal. In this case, let $N(\sigma)$ denote the minimal such $N$.

\begin{lemma}\label{lem:abs cont to vert/hori} Given any absolutely continuous closed curve $\sigma : [a,b] \rightarrow M \times (0,1]$ there exists a closed curve $\tilde\sigma : [a.b] \rightarrow M \times (0,1]$ with property-$(\star)$ such that 
$$
N(\tilde \sigma) \leq (1+L_I(\sigma))  \left( 2 +  \frac{2C_1 C_2}{T_0} \right)
$$
and 
$$
{\rm Area}_R(\tilde \sigma) \leq {\rm Area}_R(\sigma) + 2L_I(\sigma)+2.
$$
\end{lemma} 

\begin{proof} Let $N_0 : = \lfloor L_I(\sigma) \rfloor+1$. Then we can find a partition $a = t_1 < \dots < t_{N_0+1} = b$ such that 
 $$
L_I\left(\sigma|_{[t_i,t_{i+1}]}\right) < 1
 $$
 for all $i=1,\dots, N_0$. Let $h_i$ be the largest integer such that 
$$
\sigma^2([t_i, t_{i+1}]) \leq e^{-h_i T_0}. 
$$
Define $H_i : [t_i, t_{i+1}] \rightarrow M \times (0,1]$ by 
$$
H_i(t) = (\sigma^1(t), e^{-h_i T_0} ). 
$$
Then ~\ref{item:almost_orthogonal_splitting} and ~\ref{item:expansion near boundary}  imply that 
$$
L_I(H_i) \leq \frac{C_2}{C_3} L_I(\sigma|_{[t_i,t_{i+1}]}) \leq \frac{C_2}{C_3} \leq L. 
$$

Next let $V_i$ be a union of vertical curves joining $(\sigma^1(t_{i+1}), e^{-h_{i} T_0})$ to $(\sigma^1(t_{i+1}), e^{-h_{i+1} T_0})$. Notice that Lemma~\ref{lem:vertical estimates} implies that 
\begin{align*}
\abs{h_i-h_{i+1}} &  \leq 1 + \max_{s,t \in [t_i, t_{i+2}] } \frac{1}{T_0} \log \frac{\sigma^2(t)}{\sigma^2(s)} \leq 1+ \frac{C_1 C_2}{T_0} L_I( \sigma|_{[t_i, t_{i+2}]}) \\
& \leq 1+ \frac{2C_1 C_2}{T_0}.
\end{align*}
So $V_i$ consists of the union of at most $1+\frac{2C_1 C_2}{T_0}$ vertical curves.

Let $\tilde \sigma$ be the concatenation of the curves $\{H_i \cup V_i\}_{i=1}^{N_0}$. Then 
\begin{equation*}
N(\tilde \sigma) \leq N_0 \left( 2 + \frac{2C_1 C_2}{T_0} \right) \leq (1+L_I(\sigma))  \left( 2 +  \frac{2C_1 C_2}{T_0} \right). 
\end{equation*}

Next let $V_{1,i}$ be a smooth parametrization of $\{\sigma^1(t_i)\} \times [\sigma^2(t_i), e^{-h_{i} T_0}]$ and let $V_{2,i}$ be a smooth parametrization of $\{\sigma^1(t_{i+1})\} \times [\sigma^2(t_{i+1}), e^{- h_{i} T_0}]$. Lemma~\ref{lem:vertical estimates} implies that 
$$
\max_{s,t \in [t_i, t_{i+1}] } \log \frac{\sigma^2(t)}{\sigma^2(s)} \leq C_1 C_2 L_I\left(\sigma|_{[t_i,t_{i+1}]}\right) < C_1 C_2. 
$$
So using Lemma~\ref{lem:vertical estimates} again, $V_{1,i}$ and $V_{2,i}$ each have $L_I$-length at most $C_1(C_1C_2+T_0)$. 

Notice that $\sigma|_{[t_i, t_{i+1}]}$ and $V_{1,i} \cup H_i \cup V_{2,i}$ have the same endpoints and 
\begin{align*}
{\rm diam} & \left( \sigma|_{[t_i, t_{i+1}]} \cup V_{1,i} \cup H_i \cup V_{2,i}\right) \leq \frac{1}{2}L_I\left( \sigma|_{[t_i, t_{i+1}]} \cup V_{1,i} \cup H_i \cup V_{2,i}\right) \\
& \leq \frac{1}{2} + \frac{C_2}{2C_3} + C_1(C_1C_2 + T_0) \leq R. 
\end{align*}
Further, $V_i$ and $V_{2,i} \cup V_{1,i+1}$ have the same endpoints and 
\begin{align*}
{\rm diam} & \left( V_i \cup V_{1,i} \cup  V_{2,i}\right) \leq \frac{1}{2}L_I\left( V_i \cup V_{1,i} \cup  V_{2,i} \right) \leq L_I\left(  V_{1,i} \cup  V_{2,i} \right) \\
& \leq 2C_1(C_1C_2 + T_0) \leq R. 
\end{align*}
So
\begin{equation*}
{\rm Area}_R(\tilde \sigma) \leq {\rm Area}_R(\sigma) + 2N_0 \leq {\rm Area}_R(\sigma) +2L_I(\sigma) +2. \qedhere
\end{equation*}

 \end{proof} 
 
Finally, we verify the linear isoperimetric inequality for curves with property-$(\star)$.
 
 \begin{lemma}\label{lem:area bound for vert/hori curves} If $\sigma : [a,b] \rightarrow M \times (0,1]$ is a closed curve with property-$(\star)$, then 
 $$
{\rm Area}_R(\sigma) \leq N(\sigma).
$$
\end{lemma} 

\begin{proof} We apply induction on $N(\sigma)$. Fix a  closed curve $\sigma=(\sigma^1,\sigma^2) : [a,b] \rightarrow M \times (0,1]$ with property-$(\star)$. Let $h(\sigma) \in \Zb_{\geq 0}$ be the number satisfying 
$$
e^{-h(\sigma)T_0} = \min\left\{ \sigma^2(t) : t \in [a,b] \right\}. 
$$
If $h(\sigma) \leq 1$, then 
$$
{\rm diam}(\sigma) \leq {\rm diam}(M \times [e^{-T_0}, 1]) \leq R
$$
and so 
$$
{\rm Area}_R(\sigma) \leq 1 \leq N(\sigma).
$$
If $N(\sigma) \leq 3$, then
$$
{\rm diam}(\sigma) \leq \frac{1}{2} L_I(\sigma) \leq \frac{3L}{2} \leq R
$$
 and so 
 $$
{\rm Area}_R(\sigma) \leq 1 \leq N(\sigma).
$$
Hence we can assume that $h(\sigma) \geq 2$ and $N(\sigma) \geq 4$. 

Fix a partition $a=t_0 < \cdots < t_{N(\sigma)} = b$ where each $\sigma|_{[t_i,t_i+1]}$ is either vertical or horizontal. Then, by definition,
$$
\max L_I\left(\sigma|_{[t_i,t_i+1]}\right) \leq L.
$$
 
 \medskip
 
\noindent \textbf{Case 1:} Assume that no horizontal piece of $\sigma$ has image in $M \times \{e^{-h(\sigma)T_0}\}$. Then there exist $x \in M$ and two adjacent vertical pieces $V_1$ and $V_2$ of $\sigma$ both parameterizing $\{ x\} \times [ e^{-(h(\sigma)-1)T_0}, e^{-h(\sigma)T_0}]$, but with opposite orientation. Let $\tilde \sigma$ be obtained from $\sigma$ by removing $V_1$ and $V_2$. By~\eqref{eqn:vertical segment length},
 $$
\diam(V_1 \cup V_2) \leq L \leq R.
 $$
So  
 $$
{\rm Area}_R(\sigma) \leq {\rm Area}_R(\tilde \sigma) + 1
$$
and hence, by induction, 
  $$
{\rm Area}_R(\sigma) \leq N(\tilde \sigma) + 1 \leq N(\sigma)-1. 
$$

\medskip

\noindent \textbf{Case 2:} Assume $\sigma \subset  M \times \{e^{-h(\sigma)T_0}\}$, i.e. $\sigma$ consists of only horizontal pieces. Let $\tilde \sigma(t)$ be the curve $\tilde \sigma(t) = (\sigma_1(t), e^{-(h(\sigma)-2)T_0})$. Then let $H_i : = \sigma|_{[t_i,t_i+1]}$ and let $\tilde H_i = \tilde \sigma|_{[t_i,t_i+1]}$. Since $C_3 e^{\lambda T_0} >2$, if $k \leq 3$, then ~\ref{item:expansion near boundary} implies that  
$$
L_I( \tilde H_i \cup \cdots \cup \tilde H_{i+k}) \leq L.
$$
Hence $N(\tilde \sigma) \leq \frac{N(\sigma)}{4}+1 $. Let $V_i$ be the union of two vertical curves joining $\sigma(t_i)$ to $\tilde\sigma(t_i)$. By~\eqref{eqn:vertical segment length}, $L_I(V_i) \leq 2L$. Then  for $k \leq 3$, 
\begin{align*}
{\rm diam}& \left( V_i \cup H_i \cup \cdots \cup H_{i+k} \cup V_{i+k} \cup \tilde H_{i+k} \cup\cdots \cup \tilde H_i\right) \\
& \leq \frac{1}{2} L_I\left( V_i \cup H_i \cup \cdots \cup H_{i+k} \cup V_{i+k} \cup \tilde H_{i+k} \cup\cdots \cup \tilde H_i\right)\\
& \leq \frac{9}{2}L \leq R.
\end{align*}
 So 
  $$
{\rm Area}_R(\sigma) \leq {\rm Area}_R(\tilde \sigma) +  \frac{N(\sigma)}{4}+1.
$$
 Thus, by induction and the fact that $N(\sigma) \geq 4$,
  $$
{\rm Area}_R(\sigma) \leq N(\tilde \sigma)+\frac{N(\sigma)}{4}+1 \leq  \frac{N(\sigma)}{2}+2 \leq N(\sigma).
$$

\medskip

\noindent \textbf{Case 3:} Assume $\sigma$ has a horizontal piece in $M \times \{e^{-h(\sigma)T_0}\}$ and $\sigma$ is not contained in $M \times \{e^{-h(\sigma)T_0}\}$. Then at least one of the following subcases must occur.

\medskip

\noindent \textbf{Case 3(a):} There exist $x,y \in M$ and adjacent pieces $V_1$, $H$, $V_2$ where $V_1$ joins $(x,e^{-(h(\sigma)-1)T_0})$ to $(x,e^{-h(\sigma)T_0})$, $H$ joins $(x,e^{-h(\sigma)T_0})$ to $(y,e^{-h(\sigma)T_0})$, and $V_2$ joins $(y,e^{-h(\sigma)T_0})$ to $(y,e^{-(h(\sigma)-1)T_0})$. 

Let $H= (H^1, e^{-h(\sigma)T_0})$ and then define 
$$
\tilde H(t) = (H^1(t), e^{-(h(\sigma)-1)T_0}).
$$
Let $\tilde \sigma$ be obtained from $\sigma$ by adding $\tilde H$ and removing $V_1$, $H$, $V_2$. Since $C_3 e^{\lambda T_0} >2$, ~\ref{item:expansion near boundary} implies that  
$$
L_I( \tilde H) \leq \frac{1}{2} L_I(H) \leq \frac{L}{2}. 
$$
So $N(\tilde \sigma) \leq N(\sigma)-2$. Further, using ~\eqref{eqn:vertical segment length},
\begin{align*}
{\rm diam}& \left( H \cup V_1 \cup \tilde H \cup V_2 \right) \leq \frac{1}{2} L_I\left(  H \cup V_1 \cup \tilde H \cup V_2 \right) \\
& \leq 2L \leq R.
\end{align*}
 Thus, by induction, 
  $$
{\rm Area}_R(\sigma) \leq {\rm Area}_R(\tilde \sigma) +  1 \leq N(\tilde \sigma) + 1 \leq N(\sigma) -1. 
$$

\noindent \textbf{Case 3(b):} There exist $x_1,x_2,x_3 \in M$ and  adjacent pieces $V$, $H_1$, $H_2$ where $V$ joins $(x_1,e^{-(h(\sigma)-1)T_0})$ to $(x_1,e^{-h(\sigma)T_0})$, $H_1$ joins $(x_1,e^{-h(\sigma)T_0})$ to $(x_2,e^{-h(\sigma)T_0})$, and $H_2$ joins $(x_2,e^{-h(\sigma)T_0})$ to $(x_3,e^{-h(\sigma)T_0})$. 

Let $\tilde V$ be a vertical curve joining $(x_3,e^{-(h(\sigma)-1)T_0})$ to $(x_3,e^{-h(\sigma)T_0})$. Let $H_i = (H^1_i, e^{-h(\sigma)T_0})$ and then define 
$$
\tilde H_i(t) = (H_i^1(t), e^{-(h(\sigma)-1)T_0}).
$$
Let $\tilde \sigma$ be obtained from $\sigma$ by adding $\tilde H_1$, $\tilde H_2$, $\tilde V$ and removing $V$, $H_1$, $H_2$. Since $C_3 e^{\lambda T_0} >2$, ~\ref{item:expansion near boundary} implies that  
$$
L_I( \tilde H_1 \cup \tilde H_2) \leq \frac{1}{2} L_I(H_1 \cup H_2) \leq L. 
$$
So $N(\tilde \sigma) \leq N(\sigma) -1$.  Further, using ~\eqref{eqn:vertical segment length},
\begin{align*}
{\rm diam}& \left( V \cup H_1 \cup H_2  \cup \tilde H_1 \cup \tilde H_2 \cup \tilde V \right) \leq \frac{1}{2} L_I\left( V \cup H_1 \cup H_2  \cup \tilde H_1 \cup \tilde H_2\cup \tilde V \right) \\
& \leq \frac{5}{2}L \leq R.
\end{align*}
 Thus, by induction, 
  \begin{equation*}
{\rm Area}_R(\sigma) \leq {\rm Area}_R(\tilde \sigma) +  1 \leq N(\tilde \sigma) + 1 \leq N(\sigma). \qedhere
\end{equation*}

\end{proof} 

Now Lemmas~\ref{lem:rectifiable to abs cont}, ~\ref{lem:abs cont to vert/hori}, and~\ref{lem:area bound for vert/hori curves} imply that there exists $A,B > 0$ such that 
$$
{\rm Area}_R(\sigma) \leq AL(\sigma) + B 
$$
for any rectifiable closed curve. Hence  $(M \times (0,1], \dist_1)$ is Gromov hyperbolic by Theorem~\ref{thm:isomperimetric char of GH}


\section{Quasi-normal vectors and distance to the boundary}\label{sec:distance to the boundary}


In this section we introduce quasi-normal vectors, which will allow us to decompose a vector into ``quasi-normal'' and ``tangential'' components when the domain has non-smooth boundary.

Given $r,\delta > 0$ and a bounded convex domain $\Omega \subset \Kb^d$, a unit vector $\mathbf{n} \in  \Sb^{({\rm dim}_{\Rb} \Kb)d -1}$ is a \emph{$(r,\delta)$-quasi-normal vector at $x \in \partial \Omega$} if 
$$
\delta_\Omega( x + r\mathbf{n}) \geq \delta
$$
(Notice that we must have $\delta \leq r$).

\begin{example}[Smooth case] Suppose $\Omega \subset \Kb^d$ is a bounded convex domain with $\Cc^1$-smooth boundary. If $\mathbf{n} : \partial \Omega \rightarrow \Sb^{({\rm dim}_{\Rb} \Kb)d -1}$ denotes the inward pointing unit normal vector field, then $\mathbf{n}(x)$ is a quasi-normal vector at $x \in \partial\Omega$ and the constants can be chosen to be independent of $x$. 
\end{example}

\begin{example}[Non-smooth case]\label{ex:non-smooth quasi-normal vectors} Suppose $\Omega \subset \Kb^d$ is a bounded convex domain. Fix $p_0 \in \Omega$ then define $\mathbf{n} : \partial \Omega \rightarrow \Sb^{({\rm dim}_{\Rb} \Kb)d -1}$ by 
$$
\mathbf{n}(x) = \frac{p_0-x}{\norm{p_0-x}}. 
$$
Then $\mathbf{n}(x)$ is a quasi-normal vector at $x \in \partial\Omega$ and the constants can be chosen to be independent of $x$. 
\end{example}

In this section we estimate the distance to the boundary in a direction in terms a ``quasi-normal and tangential decomposition'' of the direction.

\begin{proposition}\label{prop:dist to the boundary in a direction}  Suppose $\Omega \subset \Kb^d$ is a bounded convex domain. For any $r,\delta > 0$ there exists $C > 1$ such that: If 
\begin{enumerate}
\item $x \in \partial \Omega$ and $\mathbf{n}$ is a $(\delta, r)$-quasi-normal vector at $x$, 
\item $p = x+t\mathbf{n}$ where $0 < t \leq r$, 
\item $v = s \mathbf{n}+v_0$ where $s \in \mathbb{R}$ and $(x+\Rb \cdot v_0) \cap \Omega = \emptyset$, 
\end{enumerate} 
then 
$$
\frac{1}{C} t \leq \delta_\Omega(p;\mathbf{n}) \leq t
$$
and
$$
\frac{1}{C}\delta_\Omega(p;v) \leq \frac{\norm{v}}{\frac{\abs{s}}{\delta_\Omega(p;\mathbf{n})}+ \frac{\norm{v_0}}{\delta_\Omega(p;v_0)}} \leq \delta_\Omega(p;v)
$$
In particular, 
$$
\delta_\Omega(p;v) \leq C\min\left\{  \frac{\norm{v}}{\abs{s}} \delta_\Omega(p;\mathbf{n}),  \frac{\norm{v}}{\norm{v_0}} \delta_\Omega(p;v_0)\right\}. 
$$
\end{proposition} 

For the rest of the section fix $x$, $p=x+t\mathbf{n}$, and $v = s\mathbf n + v_0$ as in the statement of the proposition.

The first estimate is a straightforward consequence of convexity. Since $x \in \partial \Omega$ and $\Omega$ contains the interior of the convex hull of $\{x\} \cup B_{\Kb^d}(x+r\mathbf{n}, \delta)$, we have 
\begin{equation}\label{eqn:normal distance versus t} 
\frac{\delta}{r} t \leq \delta_\Omega(p;\mathbf{n}) \leq t.
\end{equation}

To ease notation in the proof of the second estimate,  let
$$
\delta_1 : = \delta_\Omega(p;\mathbf{n}) \quad \text{and} \quad \delta_2 : = \delta_\Omega(p; v_0). 
$$
One inequality in the second estimate follows quickly from convexity.

\begin{lemma} $$
\delta_\Omega(p;v) \geq \frac{\norm{v}}{\frac{\abs{s}}{\delta_1}+\frac{\norm{v_0}}{\delta_2}}.
$$
\end{lemma} 

\begin{proof} By convexity, $\Omega$ contains the convex hull of $p+\delta_1(\Db \cap \Kb) \mathbf{n}$ and $p+\delta_2(\Db \cap \Kb) \frac{v_0}{\norm{v_0}}$. So if 
$$
\abs{\lambda} < \frac{1}{\frac{\abs{s}}{\delta_1}+\frac{\norm{v_0}}{\delta_2}}, 
$$
then 
$$
p+\lambda v \in\left[p+\frac{\frac{\abs{s}}{\delta_1}}{\frac{\abs{s}}{\delta_1}+\frac{\norm{v_0}}{\delta_2}} \delta_1(\Db \cap \Kb) \mathbf{n}\right] +\left[ p+ \frac{\frac{\norm{v_0}}{\delta_2}}{\frac{\abs{s}}{\delta_1}+\frac{\norm{v_0}}{\delta_2}} \delta_2(\Db \cap \Kb) \frac{v_0}{\norm{v_0}}\right] \subset \Omega. 
$$
Hence 
\begin{equation*}
\delta_\Omega(p;v) \geq \frac{\norm{v}}{\frac{\abs{s}}{\delta_1}+\frac{\norm{v_0}}{\delta_2}}.\qedhere
\end{equation*}
\end{proof}

We divide the proof of the other inequality into two cases depending on the relative sizes of $\frac{\abs{s}}{\delta_1}$ and $\frac{\norm{v_0}}{\delta_2}$.

\begin{lemma}\label{lem:upper bound from normal direction} $\delta_\Omega(p;v) \leq \frac{r}{\delta}  \frac{\delta_1}{\abs{s}}\norm{v}$. In particular, if $\frac{\abs{s}}{\delta_1} \geq \frac{1}{2} \frac{\norm{v_0}}{\delta_2}$, then 
$$
\delta_\Omega(p;v)\leq \frac{3r}{\delta} \frac{\norm{v}}{\frac{\abs{s}}{\delta_1}+ \frac{\norm{v_0}}{\delta_2}}.
$$
\end{lemma} 

\begin{proof} Notice that 
$$
p - \frac{t}{s} v= x - \frac{t}{s} v_0 \in x + \Rb \cdot v_0 \subset \Kb^d \setminus \Omega. 
$$
So by Equation~\eqref{eqn:normal distance versus t}, 
\begin{equation*}
\delta_\Omega(p;v) \leq \frac{t}{\abs{s}}\norm{v}\leq \frac{r}{\delta} \frac{\delta_1}{\abs{s}}\norm{v}. \qedhere
\end{equation*}
\end{proof}

\begin{lemma} If $\frac{\abs{s}}{\delta_1} \leq \frac{1}{2} \frac{\norm{v_0}}{\delta_2}$, then 
$$
\delta_\Omega(p;v)\leq 6 \frac{\norm{v}}{\frac{\abs{s}}{\delta_1}+ \frac{\norm{v_0}}{\delta_2}}.
$$
\end{lemma} 

\begin{proof} Fix $e^{i\theta} \in \Kb \cap \Sb^1$ such that $p+\delta_2 e^{i \theta} v_0 \in \partial \Omega$. Since $\Omega$ is convex, there exists a real affine hyperplane $H$ such that $p+\delta_2 e^{i \theta} \frac{v_0}{\norm{v_0}} \in H$ and $H \cap \Omega = \emptyset$. Then there exists $w \in \Kb^d$ such that 
$$
H = \left\{ z \in \Kb^d : {\rm Re} \ip{z-p, w} = \delta_2 \right\} \quad \text{and} \quad \Omega \subset \left\{ z \in \Kb^d : {\rm Re} \ip{z-p, w} < \delta_2 \right\}.
$$
Then 
$$
p+\frac{\delta_2}{\ip{v,w}}v \in H
$$
and so 
$$
\delta_\Omega(p;v) \leq \frac{\delta_2}{\abs{\ip{v,w}}}\norm{v}\leq \frac{\delta_2}{\abs{\ip{v_0,w}}-\abs{s}\abs{\ip{\mathbf n, w}}} \norm{v}. 
$$
Since $p+\delta_2(\Db \cap \Kb) \frac{v_0}{\norm{v_0}} \subset \Omega$ and $p+\delta_2 e^{i \theta} \frac{v_0}{\norm{v_0}} \in H$, we must have 
$$
\ip{e^{i\theta} v_0, w} = \norm{v_0}.
$$
Since $p+\delta_1(\Db \cap \Kb) \mathbf n \subset \Omega$, we must have 
$$
\abs{\ip{\mathbf n, w}} \leq \frac{\delta_2}{\delta_1} \leq \frac{\norm{v_0}}{2\abs{s}}. 
$$
Hence
\begin{equation*}
\delta_\Omega(p;v) \leq \frac{\delta_2}{\abs{\ip{v_0,w}}-\abs{s}\abs{\ip{\mathbf n, w}}} \norm{v} \leq 2\frac{\delta_2}{\norm{v_0}}\norm{v}\leq 6 \frac{\norm{v}}{\frac{\abs{s}}{\delta_1}+ \frac{\norm{v_0}}{\delta_2}}. \qedhere
\end{equation*}
\end{proof} 

Thus in all cases, 
$$
\delta_\Omega(p;v) \leq \max\left\{ 6,  \frac{3r}{\delta}\right\} \frac{\norm{v}}{\frac{\abs{s}}{\delta_1}+ \frac{\norm{v_0}}{\delta_2}}
$$
and the proof is complete.


\section{Uniformly transverse splittings}\label{sec:QH metric almost orthogonal directions} 


In this section we estimate the generalized quasi-hyperbolic length of a vector in terms a ``quasi-normal and tangential decomposition'' of the vector.

\begin{proposition}\label{prop:estimate on QH in tang/non-tang direction} Suppose $\Omega \subset \Kb^d$ is a bounded convex domain. For any $k=1,\dots, d$ and $r,\delta > 0$ there exists $C > 0$ such that: If 
\begin{enumerate}
\item $x \in \partial \Omega$ and $\mathbf{n}$ is a $(\delta, r)$-quasi-normal vector at $x$, 
\item $p = x+t\mathbf{n}$ where $0 < t \leq r$, 
\item $v = s \mathbf{n}+v_0$ where $s \in \mathbb{R}$ and $(x+\Rb \cdot v_0) \cap \Omega = \emptyset$, 
\end{enumerate} 
then 
$$
\frac{1}{C} \mathfrak{q}^{(k)}_\Omega(p;v) \leq \mathfrak{q}^{(k)}_\Omega(p;s\mathbf{n})+ \mathfrak{q}^{(k)}_\Omega(p;v_0) \leq C \mathfrak{q}^{(k)}_\Omega(p;v).
$$
Moreover, if $v_0 \neq 0$ and $(x + \Kb \cdot v_0) \cap \Omega = \emptyset$, then
$$
\delta^{(k)}_\Omega(p;v_0) \leq C \sup\left\{ \dist_{\rm Euc}(p, (p+V) \cap \partial \Omega)  : 
    V  \in \Gr_k(\Kb^d), \, v_0 \in V, \text{ and } (x+V) \cap \Omega = \emptyset
 \right\}.
$$
\end{proposition} 

We will use the following observation several times. 

\begin{observation}\label{obs:almost orthogonal vectors} There exists $\epsilon=\epsilon(k) \in (0,1)$ such that: if $u_1,\dots, u_k \in \Kb^d$ are unit vectors with 
$$
\min_{1 \leq j < l \leq k} \abs{\ip{u_j,u_l}} \leq \epsilon,
$$
then $u_1,\dots, u_k$ are linearly independent and 
$$
B_{\Kb^d}(0,\epsilon) \cap {\rm Span}_{\Kb}\{u_1,\dots, u_k\}
$$
is contained in the convex hull of $(\Kb \cap \Db) u_1 \cup \dots \cup (\Kb \cap \Db)u_k$. 

\end{observation} 

For the rest of the section fix $x$, $p=x+t\mathbf{n}$, and $v = s\mathbf n + v_0$ as in the statement of the proposition. It suffices to consider the case where $v_0 \neq 0$. We also fix $\epsilon =\epsilon(k) \in (0,1)$ satisfying Observation~\ref{obs:almost orthogonal vectors}.

Since $x \in \partial\Omega$ and  $\Omega$ contains the interior of the convex hull of $\{x\} \cup B_{\Kb^d}(x+r\mathbf{n}, \delta)$, we have 
\begin{equation}\label{eqn:t estimate}
\frac{\delta}{r} t \leq \delta_\Omega(p) \leq t.
\end{equation}

\subsection{The upper bound}

\begin{lemma} 
$\delta^{(k)}_\Omega(p;v) \geq \frac{\delta\epsilon}{r}  \min\left\{ \delta_\Omega(p;v), \delta^{(k)}_\Omega(p;v_0)\right\}$. 
\end{lemma} 

\begin{proof} We consider two cases based on the size of $s$. 

\medskip

\noindent \textbf{Case 1:} Assume $\abs{s} \geq \epsilon\norm{v}$. Notice that 
$$
p - \frac{t}{s} v= x - \frac{t}{s} v_0 \in x + \Rb \cdot v_0 \subset \Kb^d \setminus \Omega. 
$$
Hence by Equation~\eqref{eqn:t estimate},
$$
\delta_\Omega(p;v) \leq \frac{t}{\abs{s}}\norm{v} \leq \frac{t}{\epsilon} \leq \frac{r}{\delta\epsilon} \delta_\Omega(p).
$$
So
$$
\delta^{(k)}_\Omega(p;v) \geq \delta_\Omega(p) \geq \frac{\delta\epsilon}{r} \min\left\{ \delta_\Omega(p;v), \delta^{(k)}_\Omega(p;v_0)\right\}. 
$$ 

\noindent \textbf{Case 2:} Assume $\abs{s} < \epsilon\norm{v}$. Fix a $\Kb$-linear subspace $V \subset \Kb^d$ with $v_0 \in V$, $\dim_{\Kb} V = k$, and 
$$
\delta^{(k)}_\Omega(p;v_0)  = \dist_{\rm Euc}(p, (p+V) \cap \partial \Omega). 
$$
Then fix an orthonormal basis $u_1,\dots, u_k$ of $V$ with $u_1 = \frac{v_0}{\norm{v_0}}$.  Then let $W : = {\rm Span}_{\Kb}\left\{v, u_2, \dots, u_d\right\}$. 
Notice that 
$$
\delta_\Omega(p;u_j) \geq  \dist_{\rm Euc}(p, (p+V) \cap \partial \Omega) = \delta^{(k)}_\Omega(p;v_0).
$$
So $\Omega$ contains the convex hull of 
$$
p+t(\Kb \cap \Db)\frac{v}{\norm{v}} \cup t (\Kb \cap \Db)u_2 \cup \cdots \cup t(\Kb \cap \Db)u_k
$$
where 
$$
t := \min\left\{ \delta_\Omega(p;v),  \delta_\Omega^{(k)}(p;v)\right\}. 
$$
Also, 
$$
\abs{\ip{\frac{v}{\norm{v}},u_j}} = \abs{\ip{\frac{s\mathbf n}{\norm{v}}, u_j}} \leq \frac{\abs{s}}{\norm{v}} \norm{u_j} < \epsilon.
$$
So by our choice of $\epsilon > 0$, 
$$
p+B_{\Kb^d}(0,\epsilon t) \cap W \subset \Omega.
$$
So 
\begin{equation*}
\delta_\Omega^{(k)}(p;v) \geq \dist_{\rm Euc}(p, (p+W) \cap \partial \Omega) \geq \epsilon t = \epsilon \min\left\{  \delta_\Omega(p;v),  \delta^{(k)}_\Omega(p;v_0)\right\}. 
\end{equation*}
Since $\delta \leq r$, this completes the proof.
\end{proof} 

\begin{lemma}
$\mathfrak{q}^{(k)}_\Omega(p;v) \leq \frac{r^2}{\epsilon\delta^2} \left( \mathfrak{q}^{(k)}_\Omega(p;s \mathbf n)+\mathfrak{q}^{(k)}_\Omega(p;v_0)\right).$ 
\end{lemma} 

\begin{proof} By the previous lemma 
\begin{equation}\label{eqn:max of two annoying things}
\mathfrak{q}^{(k)}_\Omega(p;v) \leq \frac{r}{\epsilon\delta} \max\left\{\frac{\norm{v}}{\delta_\Omega(p;v)},  \frac{\norm{v}}{\delta_\Omega^{(k)}(p;v_0)}\right\}. 
\end{equation}
By Proposition~\ref{prop:dist to the boundary in a direction}, 
$$
\frac{\norm{v}}{\delta_\Omega(p;v)} \leq \frac{\norm{s\mathbf n}}{\delta_\Omega(p;s\mathbf n)}+\frac{\norm{v_0}}{\delta_\Omega(p;v_0)}.
$$
By definition $\delta_\Omega(p;s\mathbf n) \geq \delta_\Omega^{(k)}(p;s\mathbf n)$ and $\delta_\Omega(p;v_0) \geq \delta_\Omega^{(k)}(p;v_0)$, so 
$$
\frac{\norm{v}}{\delta_\Omega(p;v)} \leq \mathfrak{q}^{(k)}_\Omega(p;s\mathbf n)+\mathfrak{q}^{(k)}_\Omega(p;v_0).
$$
For the other term in~\eqref{eqn:max of two annoying things}, notice that Equation~\eqref{eqn:t estimate} implies that 
\begin{equation*}
\mathfrak{q}^{(k)}_\Omega(p;s \mathbf n) = \frac{\norm{s \mathbf n}}{\delta_\Omega(p;s\mathbf n)} \geq \frac{\norm{ s\mathbf n}}{t} \geq \frac{\delta}{r} \frac{\norm{ s\mathbf n}}{\delta_\Omega(p)} \geq \frac{\delta}{r}\frac{\norm{ s\mathbf n}}{\delta_\Omega^{(k)}(p;v_0)}
\end{equation*} 
and so 
\begin{equation*}
\frac{\norm{v}}{\delta_\Omega^{(k)}(p;v_0)} \leq \frac{\norm{s \mathbf n}+\norm{v_0}}{\delta_\Omega^{(k)}(p;v_0)} \leq \frac{r}{\delta} \mathfrak{q}^{(k)}_\Omega(p;s \mathbf n)+\mathfrak{q}^{(k)}_\Omega(p;v_0). \qedhere
\end{equation*}
\end{proof}

\subsection{The lower bound} Let $C = C(\Omega, r,\delta) > 1$ be as in Proposition~\ref{prop:dist to the boundary in a direction}. Then 
\begin{align*}
\delta_\Omega^{(k)}(p;v) & \leq \delta_\Omega(p;v) \leq C \frac{\norm{v}}{\abs{s}} \delta_\Omega(p; \mathbf n) \leq C \frac{\norm{v}}{\abs{s}} t.
\end{align*}
Hence by Equation~\eqref{eqn:t estimate},
\begin{align}\label{eqn:estimate for distance to boundary} 
\delta_\Omega^{(k)}(p;v) \leq C  \frac{\norm{v}}{\abs{s}}\frac{r}{\delta} \delta_\Omega(p)  \leq C  \frac{\norm{v}}{\abs{s}}\frac{r}{\delta} \delta_\Omega^{(k)}(p; \mathbf n).
\end{align}
So 
\begin{equation}\label{eqn:lower bound in terms of normal metric} 
\mathfrak{q}^{(k)}_\Omega(p;v) = \frac{\norm{v}}{\delta_\Omega^{(k)}(p;v)} \geq \frac{\delta}{Cr} \frac{\abs{s}}{\delta_\Omega^{(k)}(p; \mathbf n)} = \frac{\delta}{Cr} \mathfrak{q}^{(k)}_\Omega(p;s \mathbf n). 
\end{equation} 

Since $\delta_\Omega(x+r\mathbf n) \geq \delta$ and $(x + \Rb \cdot v_0) \cap \Omega = \emptyset$, the angle between $\Rb \cdot \mathbf n$ and $\Rb \cdot v_0$ is bounded below by a constant depending only on $r$ and $\delta$. So there exists a constant $C'> 1$, which only depends on $r$ and $\delta$, such that 
$$
\norm{v_0} \leq C'\norm{v}.
$$

\begin{lemma}\label{lem:lower bound in terms of tangential metric} 
$
\mathfrak{q}^{(k)}_\Omega(p;v_0) \leq  \frac{3CC' r}{\delta \epsilon}  \mathfrak{q}_\Omega^{(k)}(p;v). 
$
\end{lemma} 

\begin{proof} Notice that if $\abs{s} \geq \frac{\epsilon}{2}\norm{v}$, then Equation~\eqref{eqn:estimate for distance to boundary} implies that 
$$
\delta^{(k)}_\Omega(p;v)  \leq C  \frac{\norm{v}}{\abs{s}}\frac{r}{\delta} \delta_\Omega(p) \leq \frac{2C r}{ \delta\epsilon} \delta_\Omega(p).
$$
Then 
$$
\mathfrak{q}^{(k)}_\Omega(p;v_0) \leq \frac{\norm{v_0}}{\delta_\Omega(p)} \leq \frac{2C r}{\delta\epsilon }\frac{\norm{v_0}}{\delta^{(k)}_\Omega(p;v)}  \leq \frac{2CC' r}{\delta\epsilon } \frac{\norm{v}}{\delta^{(k)}_\Omega(p;v) }= \frac{2CC' r}{\delta \epsilon} \mathfrak{q}^{(k)}_\Omega(p;v).
$$
Thus we may assume that $\abs{s} < \frac{\epsilon}{2}\norm{v}$.

Fix a $\Kb$-linear subspace $V \subset \Kb^d$ with $v \in V$, $\dim_{\Kb} V = k$, and 
$$
\delta^{(k)}_\Omega(p;v)  = \dist_{\rm Euc}(p, (p+V) \cap \partial \Omega). 
$$
Then fix an orthonormal basis $u_1,\dots, u_k$ of $V$ with $u_1 = \frac{v}{\norm{v}}$.  Then let $W : = {\rm Span}_{\Kb}\left\{v_0, u_2, \dots, u_d\right\}$. Notice that 
$$
\delta_\Omega(p;u_j) \geq  \dist_{\rm Euc}(p, (p+V) \cap \partial \Omega) = \delta^{(k)}_\Omega(p;v).
$$
So $\Omega$ contains the convex hull of 
$$
p+t(\Kb \cap \Db)\frac{v_0}{\norm{v_0}} \cup  t (\Kb \cap \Db)u_2 \cup \cdots \cup t(\Kb \cap \Db)u_k
$$
where 
$$
t := \min\left\{ \delta_\Omega(p;v_0),  \delta_\Omega^{(k)}(p;v)\right\}. 
$$
Also, since $\epsilon \in (0,1)$, 
$$
\abs{\ip{\frac{v_0}{\norm{v_0}}, u_j}} = \abs{\ip{\frac{v}{\norm{v_0}}, u_j} - \ip{\frac{s \mathbf n}{\norm{v_0}}, u_j}}=\abs{ \ip{\frac{s \mathbf n}{\norm{v_0}}, u_j}} < \frac{\frac{\epsilon}{2}}{1-\frac{\epsilon}{2}}< \epsilon.
$$
So by our choice of $\epsilon > 0$, 
$$
p+B_{\Kb^d}(0, \epsilon t) \cap W \subset \Omega.
$$
Thus 
$$
\delta^{(k)}_\Omega(p;v_0) \geq \epsilon t = \epsilon  \min\left\{ \delta_\Omega(p;v_0),  \delta_\Omega^{(k)}(p;v)\right\}. 
$$

Since $C > 1$ is the constant in Proposition~\ref{prop:dist to the boundary in a direction} and $\epsilon \in (0,1)$, 
$$
 \delta_\Omega(p;v)  \leq C\frac{\norm{v}}{\norm{v_0}}\delta_\Omega(p;v_0)  < \frac{C}{1-\frac{\epsilon}{2}}\delta_\Omega(p;v_0) \leq 2C \delta_\Omega(p;v_0).
$$
By definition $ \delta_\Omega(p;v) \geq  \delta_\Omega^{(k)}(p;v)$ and so
$$
\delta^{(k)}_\Omega(p;v_0) \geq \epsilon \min\left\{ \delta_\Omega(p;v_0),  \delta_\Omega^{(k)}(p;v)\right\} \geq \frac{\epsilon}{2C} \delta_\Omega^{(k)}(p;v).
$$
Hence
\begin{equation*}
\mathfrak{q}^{(k)}_\Omega(p;v_0) = \frac{\norm{v_0}}{\delta^{(k)}_\Omega(p;v_0) } \leq \frac{2C\left(1+\frac{\epsilon}{2}\right)}{\epsilon}  \frac{\norm{v}}{\delta^{(k)}(p;v)}=\frac{3C}{\epsilon}\mathfrak{q}^{(k)}_\Omega(p;v). 
\end{equation*} 
Since $C'> 1$ and $r \geq \delta$, the lemma is true. 
\end{proof} 

Finally, Equation~\eqref{eqn:lower bound in terms of normal metric} and Lemma~\ref{lem:lower bound in terms of tangential metric} imply that 
\begin{align*}
\mathfrak{q}_\Omega^{(k)}(p;v) &  \geq \max\left\{ \frac{\delta}{Cr} \mathfrak{q}^{(k)}_\Omega(p;s \mathbf n), \frac{\delta \epsilon}{3CC' r}\mathfrak{q}^{(k)}_\Omega(p;v_0)\right\} \\
&\geq \frac{\delta \epsilon}{6CC' r}\left( \mathfrak{q}^{(k)}_\Omega(p;s \mathbf n)+ \mathfrak{q}^{(k)}_\Omega(p;v_0)\right). 
\end{align*}

\subsection{The Moreover Part} Suppose that $v_0 \neq 0$ and $(x+ \Kb \cdot v_0) \cap \Omega = \emptyset$. Let 
$$
D: = \sup\left\{ \dist_{\rm Euc}(p, (p+V) \cap \partial \Omega)  : 
    V  \in \Gr_k(\Kb^d), \, v_0 \in V, \text{ and } (x+V) \cap \Omega = \emptyset
 \right\}.
$$

Fix a $\Kb$-linear subspace $V \subset \Kb^d$ with $v \in V$, $\dim_{\Kb} V = k$, and 
$$
\delta^{(k)}_\Omega(p;v_0)  = \dist_{\rm Euc}(p, (p+V) \cap \partial \Omega). 
$$
Then fix an orthonormal basis $u_1,\dots, u_k$ of $V$ with $u_1 = \frac{v_0}{\norm{v_0}}$.   

Since $\Omega$ is convex and $(x+\Kb \cdot v_0) \cap \Omega = \emptyset$, there exists a  codimension one $\Kb$-linear subspace $H$ with $(x+H) \cap \Omega = \emptyset$ and $\Kb \cdot v_0 \subset H$. Since $(x + \Rb \cdot \mathbf{n}) \cap \Omega \neq \emptyset$, we have $H \oplus \Kb\cdot\mathbf{n} = \Kb^d$. So we can write 
$$
u_j = u_j' +\lambda_j \mathbf{n}
$$
 where $u_j' \in H$ and $\lambda_j \in \Kb$. 

Recall that $\epsilon  \in  (0,1)$ is the constant in Observation~\ref{obs:almost orthogonal vectors} and $C  > 1$ is the constant in Proposition~\ref{prop:dist to the boundary in a direction}. Fix $0 < \eta < \frac{1}{2}$ such that 
$$
 \frac{ 2\eta(1+\eta)+\eta^2}{(1-\eta)^2} < \epsilon.
 $$

\begin{lemma}\label{lem:case when abs of lambda is large} If $\abs{\lambda_j} \geq \eta$ for some $1 \leq j \leq k$, then 
$$
\delta^{(k)}_\Omega(p;v_0) \leq \frac{C r}{\delta \eta} D.
$$
\end{lemma} 

\begin{proof} Suppose $\abs{\lambda_j} \geq \eta$. Then  the ``in particular'' part of Proposition~\ref{prop:dist to the boundary in a direction} and Equation~\eqref{eqn:t estimate} imply that
\begin{align*}
\delta^{(k)}_\Omega(p;v_0)  &= \dist_{\rm Euc}(p, (p+V) \cap \partial \Omega) \leq \delta_\Omega(p;u_j) = \delta_\Omega\left(p;\frac{\bar\lambda_j}{\abs{\lambda_j}} u_j\right)  \\
& \leq \frac{C}{\abs{\lambda_j}} \delta_\Omega(p;\mathbf{n}) \leq  \frac{C}{\abs{\lambda_j}}t \leq  \frac{Cr}{\delta\abs{\lambda_j}} \delta_\Omega(p) \leq \frac{C r}{\delta \eta} D.  \qedhere
\end{align*}
\end{proof}

\begin{lemma}\label{lem:case when abs of lambda is small} If $\abs{\lambda_j} < \eta$ for all $1 \leq j \leq k$, then 
$$
\delta^{(k)}_\Omega(p;v_0) \leq \frac{2C}{\epsilon} D. 
$$
\end{lemma} 

\begin{proof} Notice that if $1 \leq j < l \leq k$, then 
$$
0 = \ip{u_j, u_l} = \ip{u_j', u_l'} +\lambda_j \ip{ \mathbf{n}, u_l'} + \bar\lambda_l \ip{u_j', \mathbf{n}} + \lambda_j \bar\lambda_l
$$
and so 
\begin{align*}
\abs{\ip{ \frac{u_j'}{\norm{u_j'}}, \frac{u_l'}{\norm{u_l'}}}} &  \leq \frac{1}{(1-\eta)^2} \abs{\ip{u_j', u_l'}}\leq \frac{ \eta\norm{u_j'} + \eta\norm{u_l'} + \eta^2}{(1-\eta^2)} \\
& \leq \frac{ 2\eta(1+\eta)+\eta^2}{(1-\eta)^2} < \epsilon. 
\end{align*}
Let $W =  {\rm Span}_{\Kb}\left\{u_1', \dots, u_k'\right\}$. Then by construction $(x+W) \cap \Omega \subset (x+H) \cap \Omega = \emptyset$ and Observation~\ref{obs:almost orthogonal vectors} implies that 
$$
\dist( p, (p+W) \cap \partial\Omega) \geq \epsilon \min_{1 \leq j \leq k} \delta_\Omega(p; u_j'). 
$$
If $\lambda_j \in \Rb$, then by the  ``in particular'' part of Proposition~\ref{prop:dist to the boundary in a direction} we have 
\begin{align*}
\delta_\Omega(p;u_j)\leq \frac{C}{\norm{u_j'}} \delta_\Omega(p;u_j') \leq \frac{C}{1-\eta} \delta_\Omega(p;u_j') \leq 2C  \delta_\Omega(p;u_j').
\end{align*} 
If $\lambda_j \notin \Rb$, then $\Kb = \Cb$.  Fix $\theta_j \in \Rb$ such that $e^{i\theta_j} \lambda_j \in \Rb$. Then by ``in particular'' part of Proposition~\ref{prop:dist to the boundary in a direction},
\begin{align*}
\delta_\Omega(p;u_j) & = \delta_\Omega(p;e^{i\theta_j}u_j)   \leq \frac{C}{\norm{e^{i \theta_j} u_j' }} \delta_\Omega(p;e^{i\theta_j}u_j' ) \leq \frac{C}{1-\eta} \delta_\Omega(p;u_j' ) \\
& \leq 2 C\delta_\Omega(p;u_j' ).
\end{align*}
Hence 
\begin{align*}
D & \geq \dist( p, (p+W) \cap \partial\Omega) \geq \epsilon \min_{1 \leq j \leq k} \delta_\Omega(p; u_j') \geq \frac{\epsilon}{2C} \min_{1 \leq j \leq k} \delta_\Omega(p;u_j) \\
& \geq  \frac{\epsilon}{2C}  \dist( p, (p+V) \cap \partial \Omega) =  \frac{\epsilon}{2C}  \delta_\Omega^{(k)}(p; v). \qedhere
\end{align*}

\end{proof} 

By Lemmas~\ref{lem:case when abs of lambda is large} and~\ref{lem:case when abs of lambda is small}, we have 
$$
\delta^{(k)}_\Omega(p;v_0) \leq  \max\left\{\frac{C r}{\delta \eta}, \, \frac{2C}{\epsilon} \right\}D 
$$
and so the ``moreover'' part is true.


\section{A sufficient condition for Gromov hyperbolicity}\label{sec:a sufficient condition for GH}


In this section we prove a strengthening of the implication (2) $\Rightarrow$ (1) in Theorem~\ref{thm:characterization in nonsmooth case}.

\begin{theorem}\label{thm:2 implies 1 in main theorem} Suppose $\Omega \subset \Kb^d$ is a bounded convex domain. Fix $p \in \Omega$. Then for $x \in \partial\Omega$ and $t \in [0,1]$, let 
$$
x_t : = (1-t)x+tp.
$$
Assume there exist $C, \lambda > 0$ such that: if $0 <  s < t \leq 1$, $x \in \partial \Omega$, $V \in \Gr_k(\Kb^d)$, and $(x+V) \cap \Omega = \emptyset$, then 
$$
 \dist_{\rm Euc}(x_s, (x_s+V) \cap \partial \Omega) \leq C \left( \frac{s}{t} \right)^{\lambda}  \dist_{\rm Euc}(x_t, (x_t+V) \cap \partial \Omega).
$$
Then  $(\Omega, \dist_\Omega^{(k)})$ is Gromov hyperbolic.
\end{theorem} 

\begin{remark} In Theorem~\ref{thm:2 implies 1 in main theorem} we only consider subspaces $V$ that are ``tangential'' to $\partial \Omega$ at $x$, while (2)  in Theorem~\ref{thm:characterization in nonsmooth case} considers all subspaces. 

\end{remark} 

The rest of the section is devoted to the proof of Theorem~\ref{thm:2 implies 1 in main theorem}. By translating we can assume that $p=0$. For each $x \in \Sb : = \Sb^{(\dim_{\Rb} \Kb)d-1}$ let $w(x) \in (0,\infty)$ be the unique positive number such that $w(x)x \in \partial \Omega$. Then define $\Phi :  \Sb\times [0,1] \rightarrow \overline\Omega$ by 
$$
\Phi(x,r) = \left( 1- \frac{r}{2} \right) w(x) x = (w(x)x)_{\frac{r}{2}}. 
$$
Since $\Omega$ is convex, the map $\Phi$ is biLipschitz and a homeomorphism onto its image. Further, $\Omega \setminus \Phi(\Sb \times (0,1])$ is relatively compact in $\Omega$. Thus to prove that $(\Omega, \dist_\Omega^{(k)})$ is Gromov hyperbolic it suffices to verify the three conditions in Theorem~\ref{thm:expanding metrics}. We start with a preliminary lemma. 

\begin{lemma}\label{lem:distance estimates for Phi} There exists $C_0 > 1$ such that: if $(x,r) \in \Sb \times (0,1]$, then
$$
\frac{1}{C_0} r \leq \delta_\Omega\big( \Phi(x,r)\big) \leq \delta_\Omega\big( \Phi(x,r); x\big) \leq C_0 r.
$$

\end{lemma} 

\begin{proof} Notice that 
$$
\delta_\Omega( \Phi(x,r);x) \leq \norm{\Phi(x,r) -w(x)x} \leq \frac{1}{2}\left(\max_{y \in \Sb} w(y)\right) r. 
$$
Further, since $\Omega$ is convex, 
$$
\delta_\Omega( \Phi(x,r)) \geq \frac{\delta_\Omega(0)}{2} r.
$$
Thus $C_0 : = \max\left\{ \frac{2}{\delta_\Omega(0)},  \frac{1}{2}\max_{y \in \Sb} w(y) \right\}$ suffices. 
\end{proof}

\begin{lemma}[{Condition~\ref{item:vertical estimate} in Theorem~\ref{thm:expanding metrics}}] There exists $C_1 > 1$ such that: if $\sigma : [a,b] \rightarrow (0,1]$ is absolutely continuous and $x \in \Sb$, then
$$
\frac{1}{C_1}\frac{\abs{\sigma'(t)}}{\sigma(t)} \leq \mathfrak{q}_\Omega^{(k)}\left(\Phi(x,\sigma(t));\frac{d}{dt} \Phi(x,\sigma(t)) \right) \leq C_1\frac{\abs{\sigma'(t)}}{\sigma(t)}
$$
for almost every $t \in [a,b]$. 
\end{lemma} 

\begin{proof} Notice that 
$$
\frac{d}{dt} \Phi(x,\sigma(t)) = \frac{\sigma'(t)}{2} w(x)x
$$
for almost every $t \in [a,b]$. Hence 
$$
\mathfrak{q}_\Omega^{(k)}\left(\Phi(x,\sigma(t));\frac{d}{dt} \Phi(x,\sigma(t)) \right)  = \frac{\abs{\sigma'(t)} w(x)}{2\delta_\Omega^{(k)}(\Phi(x,\sigma(t)); x)}. 
$$
Since $w : \Sb \rightarrow (0,\infty)$ is continuous and 
$$
\delta_\Omega\big(\Phi(x,\sigma(t))\big) \leq \delta_\Omega^{(k)}\big(\Phi(p,\sigma(t)); x\big) \leq \delta_\Omega\big(\Phi(p,\sigma(t)); x\big),
$$
the desired estimate follows from Lemma~\ref{lem:distance estimates for Phi}.
\end{proof}

We will use Proposition~\ref{prop:estimate on QH in tang/non-tang direction}  to verify the other two Conditions. To that end, for $x \in \partial \Omega$, let $\mathbf{n}_x = - \frac{x}{\norm{x}}$. Then there exists $r,\delta > 0$ such that each $\mathbf{n}_x$ is a $(r,\delta)$-quasi-normal vector at $x$. Further, relative to this choice of quasi-normal vectors, we have
 $$
 \Phi(x,r) = w(x)x + \frac{rw(x)}{2}  \mathbf{n}_{w(x)x}.
 $$
 Fix $C' > 1$ satisfying Proposition~\ref{prop:estimate on QH in tang/non-tang direction} for $k,r,\delta$.

\begin{lemma}\label{lem:tangent to boundary} If $\sigma : [a,b] \rightarrow \Sb$ is absolutely continuous and $\xi(t) := w(\sigma(t))\sigma(t)$, then $\xi$ is absolutely continuous and 
$$
\left( \xi(t) + \Rb \cdot \xi'(t) \right) \cap \Omega = \emptyset
$$
for almost every $t \in [a,b]$. 
\end{lemma}

\begin{proof} Since $x \mapsto w(x) x$ is Lipschitz, $\xi$ is absolutely continuous. Since $\Omega$ is convex and $\xi$ maps into $\partial \Omega$, we also have 
$$
\left( \xi(t) + \Rb \cdot \xi'(t) \right) \cap \Omega = \emptyset
$$
whenever $\xi'(t)$ exists. 

\end{proof}

\begin{lemma}[{Condition~\ref{item:almost_orthogonal_splitting} in Theorem~\ref{thm:expanding metrics}}] There exists $C_2 > 1$ such that: if $\sigma = (\sigma^1, \sigma^2) : [a,b] \rightarrow \Sb \times (0,1]$ is absolutely continuous, then 
\begin{align*}
\mathfrak{q}_\Omega^{(k)}& \left(\Phi(\sigma(t)); \left. \frac{d}{ds}\right|_{s=t} \Phi(\sigma^1(s),\sigma^2(t)) \right)+ \mathfrak{q}_\Omega^{(k)}\left(\Phi(\sigma(t)); \left. \frac{d}{ds}\right|_{s=t} \Phi(\sigma^1(t),\sigma^2(s))\right) \\
& \quad \quad \leq C_2\mathfrak{q}_\Omega^{(k)}\left(\Phi(\sigma(t)); \frac{d}{dt} \Phi(\sigma(t)) \right)
\end{align*}
for almost every $t \in [a,b]$. 
\end{lemma} 

\begin{proof} Let $\xi(t) := w(\sigma^1(t))\sigma^{1}(t)$. Then $\xi(t) = -w(\sigma^1(t))\mathbf n_{\xi(t)}$. Further, 
$$
\Phi(\sigma(t)) =\left( 1- \frac{\sigma^2(t)}{2} \right)\xi(t)= \xi(t) + \frac{\sigma^2(t) w(\sigma^1(t))}{2} \mathbf{n}_{ \xi(t)}
$$
and 
\begin{align*}
\frac{d}{dt} \Phi(\sigma(t)) & = \left. \frac{d}{ds}\right|_{s=t} \Phi(\sigma^1(s),\sigma^2(t))+\left. \frac{d}{ds}\right|_{s=t} \Phi(\sigma^1(t),\sigma^2(s)) \\
& =\left( 1- \frac{\sigma^2(t)}{2} \right)\xi'(t) + \frac{\sigma^{2 \prime}(t)w(\sigma^1(t))}{2}\mathbf n_{\xi(t)}
\end{align*}
for almost every $t \in [a,b]$. Then Lemma~\ref{lem:tangent to boundary} and Proposition~\ref{prop:estimate on QH in tang/non-tang direction} imply the desired estimate with $C_2 = C'$. 
\end{proof} 

Let $\lambda' : = \min\{\lambda, 1\}$. 

\begin{lemma}[{Condition~\ref{item:expansion near boundary}  in Theorem~\ref{thm:expanding metrics}}] There exists $C_3> 0$ such that: if $\sigma : [a,b] \rightarrow \Sb$ is absolutely continuous and $0 < s_1 \leq s_2 \leq 1$, then 
$$
\mathfrak{q}_\Omega^{(k)}\left(\Phi(\sigma(t),s_1);\frac{d}{dt} \Phi(\sigma(t),s_1) \right) \geq C_3 \left(\frac{s_2}{s_1}\right)^{\lambda'} \mathfrak{q}_\Omega^{(k)}\left(\Phi(\sigma(t),s_2); \frac{d}{dt}\Phi(\sigma(t),s_2) \right)
$$
for almost every $t \in [a,b]$.

\end{lemma} 

\begin{proof} For notational convenience, let $\xi(t) := w(\sigma(t))\sigma(t)$, $p_{j,t} :=\Phi(\sigma(t),s_j)$, and $v_{j,t} : = \frac{d}{dt} \Phi(\sigma(t),s_j)$ when the derivative exists. By Lemma~\ref{lem:tangent to boundary}  there exists a full measure set $A \subset [a,b]$ such that 
$$
v_{j,t}= \left( 1- \frac{s_j}{2} \right) \xi'(t).
$$
and
$$
(\xi(t) + \Rb \cdot \xi'(t)) \cap \Omega = \emptyset
$$ 
for all $t \in A$. Notice that it suffices to verify the desired estimate for $t \in A$ with $\xi'(t) \neq 0$. So let $A':= \{ t \in A : \xi'(t) \neq 0\}$.

When $\Kb = \Rb$,  the ``moreover'' part of Proposition~\ref{prop:estimate on QH in tang/non-tang direction}  implies that for every $t \in A'$ there exists $V_t \in \Gr_k(\Rb^d)$ such that $\xi'(t) \in V_t$, $(\xi(t) + V_t) \cap \Omega = \emptyset$, and 
$$
\delta_\Omega^{(k)}( p_{1,t}; \xi'(t))\leq C' \dist( p_{1,t}, (p_{1,t}+V_t) \cap \partial \Omega).
$$
Then for $t \in A'$, we have
\begin{align*}
\delta_\Omega^{(k)}( p_{1,t}; \xi'(t))& \leq C' \dist( p_{1,t}, (p_{1,t}+V_t) \cap \partial \Omega) \leq CC' \left( \frac{s_2}{s_1} \right)^{\lambda}  \dist_{\rm Euc}(p_{2,t}, (p_{2,t}+V_t) \cap \partial \Omega) \\
& \leq CC' \left( \frac{s_2}{s_1} \right)^{\lambda} \delta_\Omega^{(k)}( p_{2,t}; \xi'(t)).
\end{align*}
Hence, the lemma is true with $C_3 : = \frac{1}{CC'}$. 

The case when $\Kb = \Cb$ is more complicated. For each $t \in A'$ fix a real codimension one linear subspace $H_t$ with $\xi'(t) \in H_t$ and $(\xi(t) + H_t) \cap \Omega = \emptyset$. Since $0 \in \Omega$, we have 
$$
(\xi(t) +\Rb \cdot \xi(t)) \cap \Omega \neq \emptyset
$$
Hence 
$$
\dim_{\Rb} \left( H_t \cap \Cb \cdot \xi(t)\right) = 1.
$$
So there exists $\theta_t \in \Rb$ such that 
$$
H_t \cap \Cb \cdot  \xi(t) = \Rb \cdot e^{i\theta_t} \xi(t). 
$$
Next let $H_t^{\Cb} : = H_t \cap i H_t$, i.e. $H_t^{\Cb}$ is the maximal complex linear subspace of $H_t$. Then 
$$
H_t = H_t^{\Cb} \oplus \Rb \cdot e^{i\theta_t} \xi(t). 
$$
So we can write 
$$
\xi'(t) = u_t + a_t e^{i\theta_t} \xi(t)
$$
where $u_t \in H_t^{\Cb}$ and $a_t \in \Rb$. 

If $t \in A'$, then Proposition~\ref{prop:estimate on QH in tang/non-tang direction} implies that 
\begin{align*}
\mathfrak{q}_\Omega^{(k)}& \left(p_{1,t};v_{1,t}\right) = \mathfrak{q}_\Omega^{(k)}\left(p_{1,t};e^{-i\theta_t} v_{1,t}\right) = \left(1-\frac{s_1}{2} \right) \mathfrak{q}_\Omega^{(k)}\left(p_{1,t};e^{-i\theta_t}u_t + a_t \xi(t)\right)\\
& \geq \frac{1}{2C'} \left( \mathfrak{q}_\Omega^{(k)}\left(p_{1,t};u_t \right) + \mathfrak{q}_\Omega^{(k)}\left(p_{1,t};a_t \xi(t) \right) \right). 
\end{align*}
Likewise, 
\begin{align*}
\mathfrak{q}_\Omega^{(k)}& \left(p_{2,t};v_{2,t}\right) \leq C' \left( \mathfrak{q}_\Omega^{(k)}\left(p_{2,t};u_t \right) + \mathfrak{q}_\Omega^{(k)}\left(p_{2,t};a_t \xi(t) \right) \right). 
\end{align*}
Since $(\xi(t) + \Cb \cdot u_t) \cap \Omega = \emptyset$, the proof in the real case implies that 
$$
\mathfrak{q}_\Omega^{(k)}\left(p_{1,t};u_t \right) \geq \frac{1}{CC'} \left(\frac{s_2}{s_1}\right)^{\lambda} \mathfrak{q}_\Omega^{(k)}\left(p_{2,t} ; u_t \right).
$$
By Lemma~\ref{lem:distance estimates for Phi},
$$
\delta_\Omega^{(k)}(p_{1,t};\xi(t))\leq \delta_\Omega(p_{1,t};\xi(t)) \leq C_0 s_1
$$
and 
$$
\delta_\Omega^{(k)}(p_{2,t};\xi(t))\geq \delta_\Omega(p_{2,t}) \geq \frac{1}{C_0} s_2.
$$
So
$$
\mathfrak{q}_\Omega^{(k)}\left(p_{1,t};a_t \xi(t) \right) \geq \frac{1}{C_0} \frac{ \norm{a_t\xi(t)} }{s_1} =  \frac{1}{C_0} \frac{s_2}{s_1} \frac{ \norm{a_t\xi(t)}}{s_2} \geq \frac{1}{C_0^2} \frac{s_2}{s_1} \mathfrak{q}_\Omega^{(k)}\left(p_{2,t};a_t \xi(t) \right).
$$
Hence
\begin{equation*}
\mathfrak{q}_\Omega^{(k)} \left(p_{1,t};v_{1,t}\right)  \geq \frac{1}{2C'^2} \min\left\{\frac{1}{CC'}, \frac{1}{C_0^2}\right\}  \left(\frac{s_2}{s_1}\right)^{\lambda'} \mathfrak{q}_\Omega^{(k)} \left(p_{2,t};v_{2,t}\right). \qedhere
\end{equation*}
  \end{proof}


\section{Distance estimates for generalized quasi-hyperbolic metrics}\label{sec:distance estimates for QH}


In this section we establish bounds for the generalized quasi-hyperbolic distance. Then we use these bounds to show that the distance is proper and to construct quasi-geodesics. 

\begin{proposition}\label{prop:lower bound for distance} Suppose $\Omega \subset \Kb^d$ is a convex domain.
\begin{enumerate}
\item If $p,q \in \Omega$ and $H \subset \Kb^d$ is a $\Kb$-affine hyperplane with $H \cap \Omega = \emptyset$, then 
$$
\abs{ \log \frac{\dist_{\rm Euc}(q,H)}{\dist_{\rm Euc}(p,H)}} \leq \dist^{(k)}_\Omega(p, q).
$$
\item If $p,q \in \Omega$ are contained in a $\Rb$-affine line $L$ and $x \in L \cap \partial \Omega$, then 
$$
\abs{ \log \frac{\norm{q-x}}{\norm{p-x}}} \leq \dist^{(k)}_\Omega(p, q).
$$
\end{enumerate}
\end{proposition} 

\begin{proof} (1). By rotating and translating we can assume that $H= {\rm Span}_{\Kb}\{ e_2,\dots, e_d\}$. Then for any $z \in \Omega$ and any $v \in \Kb^d$ with $\ip{v,e_1} \neq 0$, we have 
$$
z - \frac{\ip{z,e_1}}{\ip{v,e_1}}v \in H \subset \Kb^d \setminus \Omega
$$
and so
$$
\delta_\Omega^{(k)}(z; v) \leq \delta_\Omega(z; v) \leq \abs{\frac{\ip{z,e_1}}{\ip{v,e_1}}}\norm{v}.
$$

Fix an absolutely curve $\alpha=(\alpha^1,\dots, \alpha^d) : [a,b] \rightarrow \Omega$ joining $p$ to $q$. Then 
\begin{align*}
\int_a^b &  \frac{\norm{\alpha'(t)}}{\delta_\Omega^{(k)}(\alpha(t); \alpha'(t))} dt \geq \int_a^b \frac{\abs{\alpha^{1\prime}(t)}}{\abs{\alpha^1(t)}} dt \geq\abs{ \int_a^b \frac{d}{dt} \log \abs{\alpha_1(t)} dt} \\
& =\abs{ \log\abs{\alpha_1(b)} - \log \abs{\alpha_1(a)}} = \abs{\log \abs{ \frac{\ip{q,e_1}}{\ip{p,e_1}}}} = \abs{ \log \frac{\dist_{\rm Euc}(q,H)}{\dist_{\rm Euc}(p,H)}}.
\end{align*} 

(2). Since $\Omega$ is convex, there exists a $\Kb$-affine hyperplane with $x \in H$ and $H \cap \Omega = \emptyset$. Then, since $p,q,x$ are contained in a real line,
$$
\abs{ \log \frac{\norm{q-x}}{\norm{p-x}}} =\abs{ \log \frac{\dist_{\rm Euc}(q,H)}{\dist_{\rm Euc}(p,H)}} 
$$
and so the result follows from (1). 
\end{proof} 

\begin{corollary} If $\Omega \subset \Kb^d$ is a bounded convex domain, then $\dist_\Omega^{(k)}$ is a proper and geodesic metric. \end{corollary} 

\begin{proof} By construction $\dist_\Omega^{(k)}$ is a length metric and so it suffices to show that $\dist_\Omega^{(k)}$ is proper (see Proposition~\ref{prop:HopfRinow}). Since $\mathfrak{q}_\Omega^{(k)}$ is a strongly integrable pseudo-metric (see Proposition~\ref{prop:q is strongly integrable}), $\dist_\Omega^{(k)}$ generates the standard topology. So to verify that $\dist_\Omega^{(k)}$ is a proper it suffices to fix $o \in \Omega$ and a sequence $\{p_n\}$ converging to the boundary, then show that $\dist_\Omega^{(k)}(o,p_n)$ converges to infinity. Since every boundary point is contained in a real affine hyperplane that does not intersect $\Omega$, this follows from  Proposition~\ref{prop:lower bound for distance}.

\end{proof} 

Next we use Proposition~\ref{prop:lower bound for distance} to construct quasi-geodesics.

\begin{proposition}\label{prop:building quasi geodesics}  Suppose $\Omega \subset \Kb^d$ is a convex domain. If $p \in \Omega$, $x \in \partial \Omega$, and 
$$
\epsilon \norm{x-p} \leq \dist_{\rm Euc}( p, (p+V) \cap \partial \Omega)
$$
for some $\epsilon > 0$ and some $V \in \Gr_k(\Kb^d)$ with $x-p \in V$, then the curve $\sigma: [0,\infty) \rightarrow \Omega$ defined by 
$$
\sigma(t) = x+e^{-t}(p-x)
$$
satisfies 
$$
\abs{t-s} \leq \dist^{(k)}_\Omega(\sigma(s), \sigma(t)) \leq \frac{1}{\epsilon}\abs{t-s},
$$
i.e. $\sigma$ is a $\left( \frac{1}{\epsilon}, 0\right)$-quasi-geodesic. 
\end{proposition}

\begin{proof} By Proposition~\ref{prop:lower bound for distance} we have 
$$
\dist^{(k)}_\Omega(\sigma(s), \sigma(t)) \geq \abs{ \log \frac{\norm{\sigma(s)-x}}{\norm{\sigma(t)-x}}} = \abs{t-s}.
$$
For the upper bound, notice that the definition of $\delta_\Omega^{(k)}$ and the convexity of $\Omega$ implies that 
\begin{align*}
\delta_\Omega^{(k)}(\sigma(t); \sigma'(t)) & \geq \dist_{\rm Euc}(\sigma(t), (\sigma(t)+V) \cap \partial \Omega) \geq \frac{\norm{\sigma(t)-x}}{\norm{p-x}} \dist_{\rm Euc}(p, (p+V) \cap \partial \Omega) \\
& \geq \epsilon e^{-t}\norm{p-x} = \epsilon \norm{\sigma'(t)}. 
\end{align*} 
Then, assuming $s < t$, 
\begin{align*}
\dist^{(k)}_\Omega(\sigma(s), \sigma(t)) \leq \int_s^t \frac{\norm{\sigma'(u)}}{\delta_\Omega^{(k)}(\sigma(u); \sigma'(u)) } du \leq \int_s^t \frac{1}{\epsilon} du = \frac{1}{\epsilon}\abs{t-s}. 
\end{align*}
So we also have the desired upper bound. 
\end{proof} 


\section{A necessary condition for Gromov hyperbolicity}\label{sec:a necessary condition for GH}


In this section we prove a strengthening of the implication (1) $\Rightarrow$ (2) in Theorem~\ref{thm:characterization in nonsmooth case}. 

\begin{theorem}\label{thm:1 implies 2 in main theorem} Suppose $\Omega \subset \Kb^d$ is a bounded convex domain and $(\Omega, \dist^{(k)}_\Omega)$ is Gromov hyperbolic. For any $r,\delta > 0$ there exists $C,\lambda > 0$ such that: if $x \in \partial \Omega$, $\mathbf{n}$ is a $(\delta, r)$-quasi-normal vector at $x$, $0 < t_1 \leq t_2 \leq r$, and $V \in \Gr_k(\Kb^d)$, then
$$
\dist_{\Euc}\big( x+t_1\mathbf{n}, (x+t_1 \mathbf{n} + V) \cap \partial\Omega\big) \leq C \left(\frac{t_1}{t_2} \right)^{\lambda} \dist_{\Euc}\big(x+t_2 \mathbf{n}, (x+t_2 \mathbf{n} + V) \cap \partial\Omega\big).
$$
\end{theorem} 

\begin{remark} Notice that Theorem~\ref{thm:1 implies 2 in main theorem} considers arbitrary quasi-normal vectors, while (2) in Theorem~\ref{thm:characterization in nonsmooth case} only considers a certain family of quasi-normal vectors (see Example~\ref{ex:non-smooth quasi-normal vectors}). 
\end{remark}

Notice that the theorem is equivalent to showing that there exists 
$C,\lambda > 0$ such that: if $x \in \partial \Omega$ and $\mathbf{n}$ is a $(\delta, r)$-quasi-normal vector at $x$, then 
\begin{align*}
\dist_{\Euc} & ( x+re^{-s_1} \mathbf{n}, (x+re^{-s_1} \mathbf{n} + V) \cap \partial\Omega)\\
&  \leq C e^{-\lambda(s_1-s_2)} \dist_{\Euc}( x+re^{-s_2} \mathbf{n}, (x+re^{-s_2} \mathbf{n} + V) \cap \partial\Omega)
\end{align*}
for all $0 < s_2 \leq s_1$ and $V \in \Gr_k(\Kb^d)$. To prove this exponential decay, we first prove convergence to zero and then prove that there is definite decay over any sufficiently large interval. 

\begin{proposition}\label{prop:no affine disks in boundary} 
$$
0=\lim_{\epsilon \searrow 0} \ \sup\left\{ \dist_{\Euc}( p, (p+ V) \cap \partial\Omega) : p \in \Omega, \ \dist_{\Euc}(p,\partial\Omega) \leq \epsilon, \ V \in \Gr_k(\Kb^d) \right\}.
$$
Equivalently (by convexity), for all $x \in \partial \Omega$ and $V \in \Gr_k(\Kb^d)$, the set $(x+V) \cap \partial\Omega$ has empty interior in $x+V$. 
\end{proposition}

\begin{proposition}\label{prop:definite contraction} There exist $T > 0$ and $\epsilon \in (0,1)$ such that:  if $x \in \partial \Omega$, $\mathbf{n}$ is a $(\delta, r)$-quasi-normal vector at $x$, $s_2>0$, $s_1 \geq s_2+T$, and $V \in \Gr_k(\Kb^d)$, then 
$$
\dist_{\Euc}(  x+re^{-s_1} \mathbf{n}, (x+re^{-s_1} \mathbf{n} + V) \cap \partial\Omega) \leq (1-\epsilon) \dist_{\Euc}( x+re^{-s_2} \mathbf{n}, (x+re^{-s_2} \mathbf{n} + V) \cap \partial\Omega).
$$

\end{proposition} 

Assuming the propositions we prove Theorem~\ref{thm:1 implies 2 in main theorem}. 

\begin{proof}[Proof of Theorem~\ref{thm:1 implies 2 in main theorem}] By Proposition~\ref{prop:no affine disks in boundary} there exists $T_0 > 0$ such that
$$
\dist_{\Euc}( p, (p + V) \cap \partial\Omega) \leq \delta
$$
whenever $p \in \Omega$, $\dist_{\Euc}(p,\partial \Omega) \leq e^{-T_0}$, and $V \in \Gr_k(\Kb^d)$. Let $T>0$ and  $\epsilon \in (0,1)$ be as in Proposition~\ref{prop:definite contraction}. Then let 
$$
\lambda : =- \frac{\log(1-\epsilon)}{T} \quad \text{and} \quad C := \max\left\{  \frac{{\rm diam}(\Omega)}{e^{-T_0}\delta}  e^{\lambda T_0}, e^{\lambda T}, \frac{1}{1-\epsilon}\right\}. 
$$

Fix $x \in \partial \Omega$, a $(\delta,r)$-quasi-normal vector $\mathbf{n}$ at $x \in \partial \Omega$, and $0 \leq s_2 < s_1$. Let $x_s : = x+re^{-s} \mathbf{n}$. Since $\mathbf{n}$ is a $(\delta,r)$-quasi-normal vector, $\delta_\Omega(x_0)=\delta_\Omega(x+r\mathbf{n}) \geq \delta$. So by convexity, 
$$
\delta_\Omega(x_s) \geq e^{-s} \delta. 
$$

\medskip

\noindent \textbf{Case 1:} Assume $s_1 < T_0$. Then 
\begin{align*} 
\dist_{\Euc} & ( x_{s_1}, (x_{s_1} + V) \cap \partial\Omega) \leq {\rm diam}(\Omega) \leq \frac{{\rm diam}(\Omega)}{\delta_\Omega(x_{s_2})}  \dist_{\Euc}( x_{s_2}, (x_{s_2} + V) \cap \partial\Omega) \\
& \leq  \frac{{\rm diam}(\Omega)}{e^{-T_0}\delta}  e^{\lambda T_0} e^{-\lambda(s_1-s_2)} \dist_{\Euc}( x_{s_2}, (x_{s_2} + V) \cap \partial\Omega) \\
& \leq C e^{-\lambda(s_1-s_2)} \dist_{\Euc}( x_{s_2}, (x_{s_2} + V) \cap \partial\Omega).
\end{align*}

\medskip

\noindent \textbf{Case 2:} Assume $s_1 \geq T_0$. Then by convexity 
\begin{align*}
\dist_{\Euc}&( x_{s_2}, (x_{s_2} + V) \cap \partial\Omega) \geq \min\left\{ \dist_{\Euc}( x_{0}, (x_{0} + V) \cap \partial\Omega), \, \dist_{\Euc}( x_{s_1}, (x_{s_1} + V) \cap \partial\Omega)  \right\} \\
& \geq \min\left\{ \delta, \, \dist_{\Euc}( x_{s_1}, (x_{s_1} + V) \cap \partial\Omega)\right\}= \dist_{\Euc}( x_{s_1}, (x_{s_1} + V) \cap \partial\Omega). 
\end{align*}

\noindent \textbf{Case 2(a):} Assume $s_1-s_2 \leq T$. Then
\begin{align*} 
\dist_{\Euc} & ( x_{s_1}, (x_{s_1} + V) \cap \partial\Omega) \leq\dist_{\Euc}  ( x_{s_2}, (x_{s_2} + V) \cap \partial\Omega) \\
& \leq e^{\lambda T} e^{-\lambda(s_1-s_2)} \dist_{\Euc}  ( x_{s_2}, (x_{s_2} + V) \cap \partial\Omega) \\
& \leq C  e^{-\lambda(s_1-s_2)} \dist_{\Euc}( x_{s_2}, (x_{s_2} + V) \cap \partial\Omega).
\end{align*}

\noindent \textbf{Case 2(b):} Assume $s_1-s_2 \geq T$. Then we can divide $[s_2,s_1]$ into $n$ intervals of length at least $T$ where $n \geq \frac{s_1-s_2}{T}-1$.  Then applying Proposition~\ref{prop:definite contraction} to each subinterval, we have
\begin{align*} 
\dist_{\Euc} & ( x_{s_1}, (x_{s_1} + V) \cap \partial\Omega) \leq (1-\epsilon)^n \dist_{\Euc}  ( x_{s_2}, (x_{s_2} + V) \cap \partial\Omega) \\
& \leq \frac{1}{1-\epsilon} e^{-\lambda(s_1-s_2)} \dist_{\Euc}  ( x_{s_2}, (x_{s_2} + V) \cap \partial\Omega) \\
&  \leq C e^{-\lambda(s_1-s_2)} \dist_{\Euc}  ( x_{s_2}, (x_{s_2} + V) \cap \partial\Omega). \qedhere
\end{align*}
\end{proof} 

\subsection{Proof of Proposition~\ref{prop:no affine disks in boundary}} The proof is similar to the proof of~\cite[Theorem 3.1]{ZimmerMathAnn2016}. 

Suppose not. Then there exist  $\{p_n\} \subset \Omega$ and $\{V_n\} \subset \Gr_k(\Kb^d)$ such that $\dist_{\Euc}(p_n, \partial\Omega) \rightarrow 0$ and 
$$
\rho: = \lim_{n \rightarrow \infty} \dist_{\Euc}( p_n, (p_n + V_n) \cap \partial \Omega) > 0. 
$$
Passing to a subsequence we can suppose that $p_n \rightarrow x
\in \partial \Omega$ and $V_n \rightarrow V \in \Gr_k(\Kb^d)$. Then
$$
B_{\Kb^d}(x,\rho) \cap (x+V) \subset \overline{\Omega}.
$$
Since $x \in \partial \Omega$, convexity implies that 
$$
B_{\Kb^d}(x,\rho) \cap (x+V) \subset \partial \Omega.
$$

Next fix $p \in \Omega$ and fix $y$ in the boundary of $(x+V) \cap \partial \Omega$ in $x+V$. Notice that $\norm{x-y} \geq \rho$ and hence $x \neq y$. Let 
$$
\sigma_x(t) : = x + e^{-t}(p-x) \quad \text{and} \quad \sigma_y(t) : = y + e^{-t}(p-y).
$$
By Proposition~\ref{prop:building quasi geodesics}, there exists $A > 1$ such that $\sigma_x$, $\sigma_y$ are $(A,0)$-quasi-geodesics. 

\medskip 

\noindent \textbf{Claim 1:} After possibly increasing $A > 1$, for any $T \geq 0$ the Euclidean line segment $[\sigma_x(T), \sigma_y(T)]$ can be parametrized to be a $(A,0)$-quasi-geodesic. 

\medskip

Since $B_{\Kb^d}(x,\rho) \cap (x+V) \subset \partial \Omega$, convexity implies that
$$
\dist_{\Euc}(\sigma_x(T),  (\sigma_x(T)+V) \cap \partial \Omega) \geq \min\{ \rho, \delta_\Omega(p)\}. 
$$
For $T > 0$, let $a_T$ be the point of intersection between the Euclidean ray $\overrightarrow{\sigma_x(T)\sigma_y(T)}$ and $\partial \Omega$. Since $\Omega$ is bounded, 
$$
R : = \sup_{T > 0} \norm{\sigma_x(T) - a_T} 
$$
is finite. Then by Proposition~\ref{prop:building quasi geodesics},  the Euclidean line segment $[\sigma_x(T), \sigma_y(T)]$ can be parametrized to be a $(A',0)$-quasi-geodesic where $A' : = \frac{R}{\min\{ \rho, \delta_\Omega(p)\}}$.

\medskip 

\noindent \textbf{Claim 2:} $\lim_{t \rightarrow \infty} \dist_\Omega^{(k)}(\sigma_x(t), [p,y)) = +\infty$. 

\medskip

Suppose not. Then there exist sequences $\{t_n\}$ and $\{s_n\}$ where $t_n \rightarrow +\infty$ and 
$$
\sup_{n \geq 1}  \dist_\Omega^{(k)}(\sigma_x(t_n), \sigma_y(s_n)) < + \infty. 
$$
Since $\dist_\Omega^{(k)}$ is a proper distance, we must have $s_n \rightarrow +\infty$. Let $b_n$ be the point of intersection between the Euclidean ray $\overrightarrow{\sigma_x(t_n) \sigma_y(s_n)}$ and $\partial \Omega$. Since $y$ is in the boundary of $(x+V) \cap \partial \Omega$ in $x+V$ we must have $b_n \rightarrow y$. Hence 
$$
\lim_{n \rightarrow \infty} \norm{\sigma_y(s_n)-b_n}=0 \quad \text{and} \quad \lim_{n \rightarrow \infty}\norm{\sigma_x(t_n)-b_n} = \norm{x-y} \geq \rho.
$$
Then by Proposition~\ref{prop:lower bound for distance}, 
$$
\lim_{n \rightarrow \infty}  \dist_\Omega^{(k)}(\sigma_x(t_n), \sigma_y(s_n))  \geq \lim_{n \rightarrow \infty}   \log\frac{\norm{\sigma_x(t_n)-b_n}}{\norm{\sigma_y(s_n)-b_n}} = +\infty.
$$
So we have a contradiction. 

\medskip 

\noindent \textbf{Claim 3:} For any $M > 0$, there exists $T > 0$ such that the quasi-geodesic triangle 
$$[p,\sigma_x(T)] \cup [\sigma_x(T), \sigma_y(T)] \cup [\sigma_y(T), p]$$
 is not $M$-slim. 

\medskip

By Claim 2 there exists $t_0 > 0$ such that 
$$
 \dist_\Omega^{(k)}(\sigma_x(t_0), [p,y))  > M.
 $$
Since $\dist_\Omega^{(k)}$ is a proper distance, we can fix $T > t_0$ such that 
$$
 \dist_\Omega^{(k)}(\sigma_x(t_0), [\sigma_x(T), \sigma_y(T)])  > M.
 $$
 Then 
 $$
 \dist_\Omega^{(k)}(\sigma_x(t_0), [\sigma_x(T), \sigma_y(T)] \cup [\sigma_y(T), p])  > M
 $$
 and so $[p,\sigma_x(T)] \cup [\sigma_x(T), \sigma_y(T)] \cup [\sigma_y(T), p]$ is not $M$-slim. 

Finally, Theorem~\ref{thm:shadowing property} and Claim 3 imply that $\dist_\Omega^{(k)}$ is not Gromov hyperbolic and so we have a contradiction. 

\subsection{Proof of Proposition~\ref{prop:definite contraction}}

\begin{figure} 
\begin{tikzpicture}[
  declare function={
    func(\x)= (\x < 0) * (-.3*x)   +
              (\x >= 0) * (x*x*x)
   ;
  }
]
\begin{axis}[scale=1.6, hide axis, xmin=-.9, xmax=1,ymin=-1,ymax=1]
  \addplot[smooth, domain=-.56780:.9, samples=1000] {func(x)};
    
    \draw[dashed] (0,0)--(0,.8);
    \draw[dashed] (0,.25)--(0.62996,.25);
     \draw[dashed] (0,.5)--(0.7937,.5);
     \draw[dashed] (.4,.8)--(.4,.4*.4*.4);
    
\node[fill,circle,inner sep=1.5pt,label={below left:$x_n$}] at
(0,0){};     
\node[fill,circle,inner sep=1.5pt,label={left:$x_{n,b_n}$}] at
(0,.25){};     
\node[fill,circle,inner sep=1.5pt,label={below right:$p_n'$}] at
(0.4,.25){};  
\node[fill,circle,inner sep=1.5pt,label={right:$y_n'$}] at
(0.62996,.25){};  
\node[fill,circle,inner sep=1.5pt,label={left:$x_{n,a_n}$}] at
(0,.5){};  
\node[fill,circle,inner sep=1.5pt,label={above right :$p_n$}] at
(0.4,.5){};  
\node[fill,circle,inner sep=1.5pt,label={right:$y_n$}] at
(0.7937,.5){};  
\node[fill,circle,inner sep=1.5pt,label={left:$x_{n,0}$}] at
(0,.8){};  
\node[fill,circle,inner sep=1.5pt,label={right:$q_{n}$}] at
(.4,.8){};  
\node[fill,circle,inner sep=1.5pt,label={below right:$z_{n}$}] at
(.4,.4*.4*.4){};  
\end{axis}
\end{tikzpicture}
\caption{The many points in the proof of Proposition~\ref{prop:definite contraction}}
\end{figure}
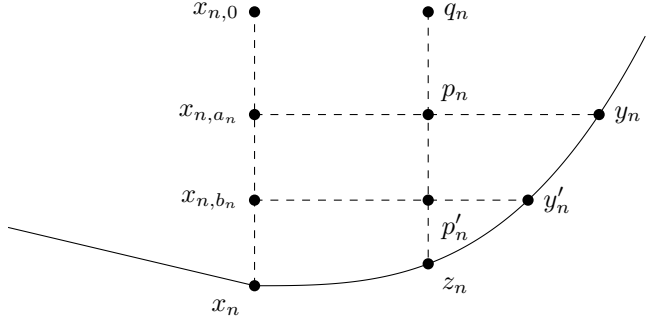 

 Suppose no such $T$ and $\epsilon$ exist. Then for every $n \in \Nb$ we can find $x_n \in \partial \Omega$, a $(r,\delta)$-quasi-normal vector $\mathbf{n}_n$ at $x_n$, positive numbers $a_n,b_n$,  and $V_n \in \Gr_k(\Kb^d)$ such that 
 $$
a_n+n <  b_n
 $$
 and  
$$
\dist_{\Euc}( x_{n,b_n}, (x_{n,b_n} + V_n) \cap \partial\Omega) \geq \left(1-\frac{1}{n} \right) \dist_{\Euc}( x_{n,a_n}, (x_{n,a_n} + V_n) \cap \partial\Omega),
$$
where we use the notation 
$$
x_{n,t} := x_n + re^{-t} \mathbf{n}_n.
$$
Using these points we will construct non-slim quasi-geodesic rectangles.

Fix $y_n \in (x_{n,a_n} + V_n) \cap \partial\Omega$ such that 
\begin{equation}\label{eqn:definition of y_n}
\norm{y_n-x_{n,a_n}} = \dist_{\Euc}( x_{n,a_n}, (x_{n,a_n} + V_n) \cap \partial\Omega).
\end{equation}
Then let
$$
p_n : = x_{n,a_n} + \left(1-\frac{2}{n} \right)(y_n - x_{n,a_n}) \quad \text{and} \quad p_n' : = x_{n,b_n} + \left(1-\frac{2}{n} \right)(y_n - x_{n,a_n}).
$$

We will show that the Euclidean line segments 
$$
[x_{n,a_n},x_{n,b_n}] \cup [x_{n,b_n}, p_n'] \cup [p_n',p_n] \cup [p_n, x_{n,a_n}]
$$
form a quasi-geodesic rectangle and then show that this family is not uniformly slim. 

Since each $\mathbf{n}_n$ is a $(r,\delta)$-quasi-normal vector, Proposition~\ref{prop:building quasi geodesics} implies that  the Euclidean line segments $[x_{n,a_n},x_{n,b_n}]$ can be parametrized to be a quasi-geodesic with constants independent of $n$. Further, Equation~\eqref{eqn:definition of y_n} and Proposition~\ref{prop:building quasi geodesics} imply that $[x_{n,a_n}, p_n] \subset [x_{n,a_n}, y_n)$ can be parametrized to be a quasi-geodesic with constants independent of $n$. Showing that the other two line segments are quasi-geodesics requires some preliminary estimates. 

Since $b_n > a_n +n > n$ and each $\mathbf{n}_n$ is a unit vector, we have 
$$
\lim_{n\rightarrow \infty} \dist_{\rm Euc}(x_{n,b_n}, \partial \Omega) \leq \lim_{n\rightarrow \infty} re^{-n} = 0. 
$$
So by Proposition~\ref{prop:no affine disks in boundary},
$$
\lim_{n \rightarrow \infty} \dist_{\Euc}( x_{n,b_n}, (x_{n,b_n} + V_n) \cap \partial\Omega)=0.
$$
Hence 
\begin{equation}\label{eqn:distance betwene y_n and x_n}
\lim_{n \rightarrow \infty} \norm{y_n-x_{n,a_n}} \leq \lim_{n \rightarrow \infty} \left(1-\frac{1}{n} \right)^{-1}\dist_{\Euc}( x_{n,b_n}, (x_{n,b_n} + V_n) \cap \partial\Omega)=0.
\end{equation}

Let $y_n' \in (x_{n,b_n}+V_n) \cap \partial \Omega$ be the point where 
$$
\{y_n'\} = (x_{n,b_n} + \Rb_{>0} \cdot (y_n-x_{n,a_n}) ) \cap \partial \Omega. 
$$

\begin{lemma}\label{lem:yn versus ynprime} After possibly passing to a tail of our sequences, 
$$
 \left(1-\frac{1}{n} \right)\norm{y_n-x_{n,a_n}} \leq \norm{y_n' - x_{n,b_n}} \leq \norm{y_n-x_{n,a_n}}
$$
for all $n \geq 1$. Hence $[p_n, p_n'] \subset \Omega$ for all $n \geq 1$. 
\end{lemma} 

\begin{proof} First notice that 
\begin{align*}
& \norm{y_n'-x_{n,b_n}} \geq \dist_\Omega( x_{n,b_n}, (x_{n,b_n} + V_n) \cap \partial\Omega) \\
& \quad \geq \left(1-\frac{1}{n} \right) \dist_\Omega( x_{n,a_n}, (x_{n,a_n} + V_n) \cap \partial\Omega) = \left(1-\frac{1}{n} \right)\norm{y_n-x_{n,a_n}}. 
\end{align*}

Since $\mathbf{n}_n$ is a $(\delta,r)$-quasi-normal vector,
$$
\delta_\Omega( x _n+ r \mathbf{n}_n) \geq \delta.
$$
Then convexity implies that
$$
\norm{y_n-x_{n,a_n}} \geq \min\{ \delta, \norm{y_n' - x_{n,b_n}}\}. 
$$
Equation~\eqref{eqn:distance betwene y_n and x_n} implies that $\norm{y_n-x_{n,a_n}} < \delta$ for $n$ sufficiently large. So 
$$
\norm{y_n-x_{n,a_n}} \geq \norm{y_n' - x_{n,b_n}} 
$$
for $n$ sufficiently large.
\end{proof}

\begin{lemma} The Euclidean line segment $[x_{n,b_n}, y_n')$ can be parametrized to be a quasi-geodesic with constants independent of $n$. 
\end{lemma} 

\begin{proof} By Lemma~\ref{lem:yn versus ynprime},
\begin{align*}
\norm{y_n'-x_{n,b_n}} & \leq \norm{y_n- x_{n,a_n}}  = \dist_{\Euc}( x_{n,a_n}, (x_{n,a_n} + V_n) \cap \partial\Omega)\\
&  \leq  \left(1-\frac{1}{n} \right)^{-1}\dist_{\Euc}( x_{n,b_n}, (x_{n,b_n} + V_n) \cap \partial\Omega).
\end{align*}
Thus 
$$
\sup_{n \geq 1} \, \frac{\norm{y_n'-x_{n,b_n}}}{\dist_{\Euc}( x_{n,b_n}, (x_{n,b_n} + V_n) \cap \partial\Omega)} < +\infty.
$$
So Proposition~\ref{prop:building quasi geodesics} implies that  $[x_{n,b_n}, y_n')$ can be parametrized to be a quasi-geodesic with constants independent of $n$. 
\end{proof}

\begin{lemma}\label{lem:pnpn' QI geodesic} After possibly passing to a tail of our sequences, the Euclidean line segment $[p_n, p_n']$ can be parametrized to be a quasi-geodesic with constants independent of $n$. Moreover, there exists $A > 1$ such that: if $t \in [a_n, a_n+n/2]$, then 
$$
\dist_\Omega^{(k)}\left(p_n, x_{n,t} +  \left(1-\frac{2}{n} \right)(y_n - x_{n,a_n}) \right) \leq A (t-a_n). 
$$
\end{lemma} 

\begin{proof} Let $q_n : = x_{n,0} +  \left(1-\frac{2}{n} \right)(y_n - x_{n,a_n})$. By Equation~\eqref{eqn:distance betwene y_n and x_n}, 
$$
\lim_{n \rightarrow \infty} \dist_{\rm Euc}( x_{n,0}, q_n) = 0.
$$
Since $\mathbf{n}_n$ is a $(\delta,r)$-quasi-normal vector, $\delta_\Omega( x_{n,0} ) \geq \delta$. So there exists $N > 0$ such that $q_n \in \Omega$ and $\delta_\Omega(q_n) \geq \delta/2$ when $n \geq N$.

For $n \geq N$, let $z_n \in \partial \Omega$ be the point where 
$$
\{z_n\} = (q_n - \Rb_{>0} \cdot \mathbf{n}_n) \cap \partial \Omega. 
$$
Then $[p_n, p_n'] \subset [q_n, z_n)$.  Further, since $\Omega$ is bounded, 
$$
\sup_{n \geq N} \frac{\norm{q_n-z_n}}{\delta_\Omega(q_n)} \leq \frac{\delta}{2} \sup_{n \geq N} \norm{q_n-z_n} < +\infty.
$$
So by Proposition~\ref{prop:building quasi geodesics} there exists $A_0 > 0$ such that 
$$
t \mapsto z_n + e^{-t_n}(q_n-z_n)
$$
is a $(A_0,0)$-quasi-geodesic when $n \geq N$.

Now fix $n \geq N$. For $t \in [a_n, a_n+n/2]$ let 
$$
q_{n,t}: = x_{n,t} +  \left(1-\frac{2}{n} \right)(y_n - x_{n,a_n}) \in [p_n,p_n']. 
$$
Then 
$$
\dist_\Omega^{(k)}\left(p_n, q_{n,t} \right) \leq A_0 \log \frac{\norm{p_n-z_n}}{\norm{q_{n,t}-z_n}} = A_0 \log \frac{\norm{p_n-p_n'}+\norm{p_n'-z_n}}{\norm{q_{n,t}-p_n'}+\norm{p_n'-z_n}}. 
$$
Notice that 
$$
f_{n,t}(s) =  \frac{\norm{p_n-p_n'}+s}{\norm{q_{n,t}-p_n'}+s} = 1 + \frac{\norm{p_n-p_n'}-\norm{q_{n,t}-p_n'}}{\norm{q_{n,t}-p_n'}+s}
$$
is decreasing on $[0,\infty)$. Hence 
\begin{align*}
\dist_\Omega^{(k)}\left(p_n, q_{n,t} \right) &  \leq A_0\log f_{n,t}(\norm{p_n'-z_n}) \leq A \log f_{n,t}(0)= A_0 \log \frac{\norm{p_n-p_n'}}{\norm{q_{n,t}-p_n'}} \\
& = A_0 \log \frac{e^{-a_n}-e^{-b_n}}{e^{-t}-e^{-b_n}}.
\end{align*}
Now the function $g_n(t) = A_0 \log \frac{e^{-a_n}-e^{-b_n}}{e^{-t}-e^{-b_n}}$ vanishes at $t=a_n$ and has derivative 
$$
g_n'(t) = \frac{A_0}{1-e^{-(b_n-t)}}.
$$
Since $b_n > a_n +n$, when $t \in [a_n, a_n+n/2]$ we have
$$
g_n'(t) \leq \frac{A_0}{1-e^{-n/2}} \leq \frac{A_0}{1-e^{-1/2}}.
$$
Hence when $t \in [a_n, a_n+n/2]$,
\begin{equation*}
\dist_\Omega^{(k)}\left(p_n, q_{n,t} \right) \leq \frac{A_0}{1-e^{-1/2}}(t-a_n). \qedhere
\end{equation*}
\end{proof} 

Thus  
$$
[x_{n,a_n}, x_{n,b_n}] \cup [x_{n,b_n}, p_n'] \cup [p_n',p_n] \cup [p_n, x_{n,a_n}]
$$
is a quasi-geodesic rectangle with constants independent of $n$. Our next aim is to show that this family is not uniformly slim. To establish the necessary distance estimates we will need a preliminary lemma.

Since $\Omega$ is convex, there is an $\Kb$-affine hyperplane $H_n$ with $x_n \in H_n$ and $H_n \cap \Omega = \emptyset$. Then there exists a unit vector $v_n$ such that 
$$
H_n = \{ z \in \Kb^d : \ip{z-x_n, v_n} = 0\}
$$
and 
\begin{equation}\label{eqn:half plane inclusion} 
\Omega \subset  \{ z \in \Kb^d : {\rm Re} \ip{z-x_n, v_n} > 0\}.
\end{equation}

\begin{lemma}\label{lem:distance to Hn} After possibly passing to a tail of our sequences, there exists $c> 0$ such that:
\begin{enumerate}
\item If $p \in [x_{n,a_n}, y_n)$, then 
$$
\dist_{\Euc}(p, H_n) \geq c e^{-a_n}.
$$
\item If $t \in [a_n,b_n]$ and $\lambda \in [0,1-1/n]$, then 
$$
\dist_{\Euc}(x_{n,t} + \lambda (y_n-x_{n,a_n}), H_n) \leq \frac{1}{c} e^{-t}. 
$$
\end{enumerate}
\end{lemma} 

\begin{proof} 
Since $\mathbf{n}_n$ is a $(\delta,r)$-quasi-normal vector, $\delta_\Omega( x _n+ r \mathbf{n}_n) \geq \delta.$
So $x_n + r\mathbf{n}_n - \delta v_n \in \overline\Omega$. Hence~\eqref{eqn:half plane inclusion}  implies that
\begin{equation}\label{eqn:vn component of n(x_n)}
{\rm Re} \ip{ \mathbf{n}_n,v_n} \geq \frac{\delta}{r}. 
\end{equation}
Pick $e^{i\theta} \in \Kb \cap \mathbb{S}^1$ such that 
$$
\ip{e^{i\theta}(y_n-x_{n,a_n}), v_n} = -\abs{ \ip{y_n-x_{n,a_n}, v_n}}. 
$$
By hypothesis, 
\begin{align*}
\dist_{\rm Euc}  (x_{n,b_n}, (x_{n,b_n} + V_n) \cap \partial \Omega) & \geq \left( 1 - \frac{1}{n}\right)\dist_{\rm Euc}(x_{n,a_n}, (x_{n,a_n} + V_n) \cap \partial \Omega)\\
& =  \left( 1 - \frac{1}{n}\right)\norm{y_n-x_{n,a_n}}.
\end{align*}
Thus
$$
x_{n,b_n} + e^{i\theta}\left(1-\frac{1}{n}\right)(y_n-x_{n,a_n}) \in \overline{\Omega}.
$$
So by~\eqref{eqn:half plane inclusion},
\begin{align*}
0 & \leq {\rm Re} \ip{x_{n,b_n} + e^{i\theta}\left(1-\frac{1}{n}\right)(y_n-x_{n,a_n})-x_n ,v_n} \\
& = re^{-b_n} {\rm Re} \ip{\mathbf{n}_n, v_n} - \left(1-\frac{1}{n}\right)\abs{ \ip{y_n-x_{n,a_n}, v_n}}.
\end{align*}
Hence 
\begin{equation}\label{eqn:yn-xn component of n(x_n)}
\abs{ \ip{y_n-x_{n,a_n}, v_n}} \leq \frac{re^{-b_n}}{1-\frac{1}{n}} {\rm Re}\ip{\mathbf{n}_n, v_n} \leq \frac{re^{-b_n}}{1-\frac{1}{n}} \abs{\ip{\mathbf{n}_n, v_n}}. 
\end{equation}

(1). Fix $p \in [x_{n,a_n}, y_n)$. Then $p = x_{n,a_n} + \lambda (y_n -  x_{n,a_n} )$ for some $\lambda \in [0,1)$. Then by Equations~\eqref{eqn:yn-xn component of n(x_n)} and~\eqref{eqn:vn component of n(x_n)},
\begin{align*}
\dist_{\Euc}(p, H_n) & = \abs{\ip{p-x_n, v_n}} = \abs{\ip{r e^{-a_n} \mathbf{n}_n + \lambda (y_n -  x_{n,a_n} ), v_n}} \\
& \geq re^{-a_n} \abs{\ip{\mathbf{n}_n ,v_n}} - \abs{\ip{y_n -  x_{n,a_n} , v_n}} \\
& \geq \left( re^{-a_n} - \frac{re^{-b_n}}{1-\frac{1}{n}}\right) \abs{\ip{\mathbf{n}_n ,v_n}} \geq \left( re^{-a_n} - \frac{re^{-b_n}}{1-\frac{1}{n}}\right) \frac{\delta}{r}. 
\end{align*}
Since $b_n > a_n + n$, this completes the proof of (1). 

(2). Fix $t \in [a_n,b_n]$ and $\lambda \in [0,1-1/n]$. By Equation~\eqref{eqn:yn-xn component of n(x_n)},
\begin{align*}
\dist_{\Euc} & (x_{n,t} + \lambda (y_n-x_{n,a_n}), H_n) =\abs{\ip{x_{n,t} + \lambda (y_n-x_{n,a_n})-x_n, v_n}} \\
& = \abs{\ip{re^{-t} \mathbf{n}_n + \lambda (y_n-x_{n,a_n}), v_n}} \\
& \leq re^{-t} \abs{\ip{\mathbf{n}_n ,v_n}} + \frac{re^{-b_n}}{1-\frac{1}{n}} \abs{\ip{\mathbf{n}_n, v_n}}.
\end{align*}
Since $t \leq b_n$ and $\mathbf{n}_n, v_n$ are unit vectors, this completes the proof of (2).

\end{proof} 

Recall that $A > 1$ was introduced in Lemma~\ref{lem:pnpn' QI geodesic}. By Proposition~\ref{prop:building quasi geodesics} we can increase $A > 1$ and also assume that 
\begin{equation}\label{eqn:upper bound on x_nt distances}
\dist_\Omega^{(k)}(x_{n,t}, x_{n,s}) \leq A\abs{t-s}
\end{equation} 
for all $n \geq 1$ and $s,t \in [0,\infty)$.

\begin{lemma}\label{lem:non slim QG rectangles} For every $M > 0$, the quasi-geodesic rectangle
$$
[x_{n,a_n}, x_{n,b_n}] \cup [x_{n,b_n}, p_n'] \cup [p_n',p_n] \cup [p_n, x_{n,a_n}]
$$
is not $M$-slim when $n$ is sufficiently large (i.e. there is a point on one of the segments which is not in the $M$-neighborhood of any of the other segments). 
\end{lemma} 

\begin{proof} By Proposition~\ref{prop:lower bound for distance} and the definition of $p_n$, 
$$
\dist^{(k)}_\Omega(x_{n,a_n}, p_n) \geq \abs{ \log \frac{\norm{x_{n,a_n}-y_n}}{\norm{p_n-y_n}}} = \abs{\log \frac{n}{2}}. 
$$
So for each $n$ we can pick $u_n \in [x_{n,a_n}, p_n]$ such that 
$$
\lim_{n \rightarrow \infty} \dist_\Omega^{(k)}(u_n, x_{n,a_n}) = +\infty = \lim_{n \rightarrow \infty} \dist_\Omega^{(k)}(u_n, p_n).
$$

Fix $M > 0$. Then fix $N_1 \geq 1$ such that: if $n \geq N_1$, then 
\begin{equation}\label{eqn:defn of N1}
\min\left\{\dist_\Omega^{(k)}(u_n, x_{n,a_n}) , \dist_\Omega^{(k)}(u_n, p_n)\right\} > (A+1)M +2A\log c.
\end{equation} 
Then let 
$$
N_2 : = \max\{ N_1, \, M-\log c, \, 2M-2\log c\}. 
$$

\medskip

\noindent \textbf{Claim 1:} If $n \geq N_2$, then 
$$
\dist_\Omega^{(k)}(u_n, [x_{n,b_n}, p_n']) > M.
$$

\medskip

\noindent By Lemma~\ref{lem:distance to Hn} and Proposition~\ref{prop:lower bound for distance},
$$
\dist_\Omega^{(k)}(u_n, [x_{n,b_n}, p_n']) \geq \log \frac{ c e^{-a_n}}{c^{-1}e^{-b_n}} = b_n -a_n + 2\log c   > n +2\log c \geq M.
$$

\medskip

\noindent \textbf{Claim 2:} If $n \geq N_1$, then 
$$
\dist_\Omega^{(k)}(u_n, [x_{n,a_n}, x_{n,b_n}]) > M.
$$

\medskip

\noindent Fix $t \in [a_n,b_n]$. If $t > a_n+M-2\log c$, then by  Lemma~\ref{lem:distance to Hn} and Proposition~\ref{prop:lower bound for distance},
$$
\dist_\Omega^{(k)}(u_n, x_{n,t}) \geq \log \frac{ c e^{-a_n}}{c^{-1}e^{-t}} = t-a_n + 2\log c   > M.
$$
If $t \leq a_n+M -2\log c$, then by Equations~\eqref{eqn:upper bound on x_nt distances} and~\eqref{eqn:defn of N1}, 
\begin{align*}
\dist_\Omega^{(k)}(u_n, x_{n,t})& \geq \dist_\Omega^{(k)}(u_n, x_{n,a_n})-\dist_\Omega^{(k)}(x_{n,a_n}, x_{n,t})  \\
& > (A+1)M +2A\log c-A(t-a_n) \geq M.
\end{align*}

\medskip

\noindent \textbf{Claim 3:} If $n \geq N_2$, then 
$$
\dist_\Omega^{(k)}(u_n, [p_n, p_n']) > M.
$$ 

\medskip

\noindent Fix $q \in [p_n, p_n']$. Then $q = x_{n,t}+ \left(1-\frac{2}{n} \right)(y_n - x_{n,a_n})$ for some $t \in [a_n,b_n]$.  If $t > a_n+M-2\log c$, then by  Lemma~\ref{lem:distance to Hn} and Proposition~\ref{prop:lower bound for distance},
$$
\dist_\Omega^{(k)}(u_n, q) \geq \log \frac{ c e^{-a_n}}{c^{-1}e^{-t}} = t-a_n + 2\log c   > M.
$$
 If $t \leq a_n+M-2\log c$, then 
 $$
 t \leq a_n + \frac{N_2}{2} \leq a_n + \frac{n}{2}. 
 $$
 So by Equation~\eqref{eqn:defn of N1} and Lemma~\ref{lem:pnpn' QI geodesic}
\begin{align*}
\dist_\Omega^{(k)}(u_n, q)& \geq \dist_\Omega^{(k)}(u_n, p_n)-\dist_\Omega^{(k)}(p_n, q)  \\
& > (A+1)M +2A\log c-A(t-a_n) \geq M. 
\end{align*}

Hence when $n \geq N_2$, the quasi-geodesic rectangle 
$$
[x_{n,a_n}, x_{n,b_n}] \cup [x_{n,b_n}, p_n'] \cup [p_n',p_n] \cup [p_n, x_{n,a_n}]
$$
is not $M$-slim
\end{proof} 

\begin{lemma} $(\Omega, \dist_\Omega^{(k)})$ is not Gromov hyperbolic. \end{lemma} 

\begin{proof} For each $n$, fix geodesic segments joining the points in $\{x_{n,a_n}, x_{n,b_n}, p_n', p_n\}$. Using these geodesic segments, let $Q_n$ be the geodesic rectangle with vertices $x_{n,a_n}, x_{n,b_n}, p_n', p_n$; let $T_{n}$ be the geodesic triangle with vertices $x_{n,a_n}, x_{n,b_n}, p_n'$; and let $T_n'$ be the geodesic triangle with vertices $x_{n,a_n}, p_n', p_n$.  

By  Theorem~\ref{thm:shadowing property} each side of $Q_n$ has uniformly bounded Hausdorff distance to the corresponding side in the quasi-geodesic rectangle 
$$
[x_{n,a_n}, x_{n,b_n}] \cup [x_{n,b_n}, p_n'] \cup [p_n',p_n] \cup [p_n, x_{n,a_n}].
$$
So Lemma~\ref{lem:non slim QG rectangles} implies that for every $M > 0$ the geodesic rectangle $Q_n$ is not $M$-slim when $n$ is sufficiently large. Then, for such $n$, at least one of $T_n$ or $T_n'$ is not $M/2$-slim. Thus $(\Omega, \dist_\Omega^{(k)})$ is not Gromov hyperbolic.
\end{proof}


\section{Gromov hyperbolicity in the smooth case}\label{sec:the smooth case}


In this section we prove Theorem~\ref{thm:characterization in smooth case}, which we restate here.

\begin{theorem}\label{thm:smooth case} Suppose $\Omega \subset \Kb^d$ is a smoothly bounded convex domain. Then the generalized $k$-quasi-hyperbolic metric on $\Omega$ is Gromov hyperbolic if and only if every $\Kb$-affine $k$-plane has finite order contact with $\partial \Omega$. 
\end{theorem}

This section is divided into four subsections. In Section~\ref{sec:definition of finite type} we recall the definition of finite order contact, in Section~\ref{sec:GH implies finite type} we prove the $(\Rightarrow)$ direction of the theorem, in Section~\ref{sec:estimates on polynomials} we establish some estimates for polynomials, and in Section~\ref{sec:finite type implies GH} we prove the $(\Leftarrow)$ direction of the theorem.

\subsection{The definitions}\label{sec:definition of finite type} Given a smooth function $f : U \rightarrow \Rb$ defined in a neighborhood $U$ of $0$ in $\Kb^k$, let $\nu(f)\in \Nb \cup \{\infty\}$ denote the order of vanishing of $f$ at $0$, i.e. if $\Kb = \Rb$, then 
$$
\nu(f) = \sup\left\{ m : \frac{\partial^{\abs{\alpha}} f}{\partial x^{\alpha}}(0) = 0 \text{ when } \abs{\alpha} < m \right\}
$$
and if $\Kb = \Cb$, then 
$$
\nu(f) = \sup\left\{ m : \frac{\partial^{\abs{\alpha}+\abs{\beta}} f}{\partial z^{\alpha}\partial \bar z^{\beta} }(0) = 0 \text{ when } \abs{\alpha}+\abs{\beta} < m \right\}.
$$

Given a smoothly bounded domain $\Omega \subset \Kb^d$, fix a smooth defining function $r : \Kb^d\rightarrow \Rb$ i.e. $\nabla r \neq 0$ near $\partial \Omega$ and $\Omega = \{ z : r(z) < 0\}$. Then for $p \in \partial\Omega$, let 
$$
L_k(\Omega, p) : = \sup\left\{ \nu( r \circ T):    T : \Kb^k \rightarrow \Kb^d \text{ is an $\Kb$-affine embedding with $T(0)=p$}\right\}. 
$$
Recall that if $\tilde r : \Kb^d \rightarrow \Rb$ is another defining function for $\Omega$, then there exists a neighborhood $\Oc$ of $\partial \Omega$ and a smooth function $\lambda : \Oc \rightarrow \Rb$ such that $\tilde r = \lambda r$ on $\Oc$ and $\lambda$ is non-vanishing on $\partial\Omega$. Hence $L_k(\Omega, p)$ is independent of the choice of defining function. 

\begin{definition}\label{defn:finite order contact} Suppose $\Omega \subset \Kb^d$ is a smoothly bounded domain. If 
$$
L_k(\Omega) := \sup_{p \in \partial \Omega} L_k(\Omega, p) \in \{2,3,\dots\} \cup \{+\infty\}
$$
is finite, then we say that \emph{every $\Kb$-affine $k$-plane has finite order contact with $\partial \Omega$.}
\end{definition} 

\begin{observation} If $\Omega \subset \Kb^d$ is a smoothly bounded domain, then there exists an $\Kb$-affine embedding $T : \Kb^k \rightarrow \Kb^d$ where 
$$
L_k(\Omega) = \nu(r \circ T). 
$$
\end{observation} 

\begin{proof} 
Clearly such an affine map exists when $L_k(\Omega)$ is finite. So assume $L_k(\Omega) = \infty$. Fix $\{p_n\} \subset \partial\Omega$ such that $L_k(\Omega, p_n) \rightarrow +\infty$. Then fix $\{T_n\}$ such that $T_n : \Kb^k \rightarrow \Kb^d$ is an $\Kb$-affine embedding with $T_n(0)=p_n$ and $\nu(r \circ T_n) \rightarrow +\infty$. Using the singular value decomposition of the linear part of $T$, for each $n$ we can pick an invertible matrix $A_n \in \mathsf{GL}(k,\Kb)$ such that 
$$
\norm{T_n A_n(x) - T_n A_n(y)} = \norm{x-y}
$$
for all $x,y \in \Kb^k$. Then passing to a subsequence we can suppose that $\{T_n A_n\}$ converges to an $\Kb$-affine map $T : \Kb^k \rightarrow \Kb^d$. Then $T$ is an $\Kb$-affine embedding with $T(0) \in \partial\Omega$ and 
\begin{equation*}
\nu(r \circ T) \geq \limsup_{n \rightarrow \infty} \nu(r \circ (T_nA_n)) = \limsup_{n \rightarrow \infty} \nu(r \circ T_n) = +\infty.  \qedhere
\end{equation*}
\end{proof}

\subsection{Gromov hyperbolicity implies finite order contact}\label{sec:GH implies finite type} Suppose $(\Omega, \dist_\Omega^{(k)})$ is Gromov hyperbolic and suppose for a contradiction that $L_k(\Omega)=+\infty$.

After translating we can assume that $0 \in \partial \Omega$ and $L_k(\Omega,0)=+\infty$. Then after rotating, we can assume that there exist
\begin{enumerate} 
\item an open neighborhood $\Oc$ of $0$ in $V:={\rm Span}_{\Kb}\{e_2,\dots, e_{k+1}\}$, 
\item a smooth function $f : \Oc \rightarrow [0,\infty)$
\end{enumerate} 
such that 
$$
\Omega \cap \big( (-2\epsilon,2\epsilon) \times \Oc \big)  = \{ (x, z_1,\dots, z_k,0,\dots,0) \in \Kb^d : x > f(z_1,\dots, z_k) \}
$$
and 
\begin{equation*}
\lim_{z \rightarrow 0} \frac{ f(z)}{\norm{z}^n} = 0 
\end{equation*}
for all $n \geq 0$. Then
$$
\lim_{t \searrow 0} \frac{ \dist_{\rm Euc}( te_1, (te_1+V) \cap \partial \Omega)}{t^{1/n}} = +\infty
$$
for all $n \geq 1$. 

For each $n \geq 1$, fix
$$
C_n > \max\left\{ n, \frac{ \dist_{\rm Euc}( \epsilon e_1, (\epsilon e_1+V) \cap \partial \Omega)}{\epsilon^{1/n}} \right\}. 
$$
Then fix $t_n \in (0,\epsilon)$ such that 
$$
 \frac{ \dist_{\rm Euc}( t_n e_1, (t_n e_1+V) \cap \partial \Omega)}{t_n^{1/n}} = C_n
$$
and 
$$
 \frac{ \dist_{\rm Euc}( t e_1, (t e_1+V) \cap \partial \Omega)}{t^{1/n}} < C_n
$$
for all $t \in (t_n, \epsilon]$. Then 
$$
\dist_{\rm Euc}( t_n e_1, (t_n e_1+V) \cap \partial \Omega) \geq \left(\frac{t_n}{t}\right)^{1/n}\dist_{\rm Euc}( t e_1, (t e_1+V) \cap \partial \Omega)
$$
for all $n \geq 1$ and $t \in (t_n, \epsilon]$. Also notice that since $C_n \rightarrow +\infty$ and $\Omega$ is bounded, we must have 
$$
\lim_{n \rightarrow \infty} t_n = 0. 
$$

Since $(\Omega, \dist_\Omega^{(k)})$ is Gromov hyperbolic, by Theorem~\ref{thm:1 implies 2 in main theorem}  there exist $C,\lambda > 0$ such that
$$
\dist_{\rm Euc}( t_n e_1, (t_n e_1+V) \cap \partial \Omega) \leq C\left(\frac{t_n}{t}\right)^{\lambda}\dist_{\rm Euc}( t e_1, (t e_1+V) \cap \partial \Omega)
$$
for all $n \geq 1$ and $t \in (t_n, \epsilon]$. Then 
$$
\left(\frac{t_n}{t}\right)^{1/n} \leq C\left(\frac{t_n}{t}\right)^{\lambda}
$$
for all $n \geq 1$ and $t \in (t_n, \epsilon]$, which is clearly impossible since $t_n \rightarrow 0$. 

\subsection{Estimates on polynomials}\label{sec:estimates on polynomials} Before starting the proof that finite order contact implies Gromov hyperbolicity, we establish some estimates for polynomials. 

Given a polynomial $P : \Rb^d \rightarrow \Rb$ of degree at most $L$ we can write $P(x) = \sum_{\abs{\alpha} \leq L} c_\alpha x^{\alpha}$ where $\alpha = (\alpha_1,\dots, \alpha_d) \in \Zb_{\geq 0}^d$, $\abs{\alpha} = \alpha_1+\cdots + \alpha_d$,  and $x^{\alpha} = x_1^{\alpha_1} \cdots x_d^{\alpha_d}$. Then we define 
$$
\norm{P}  := \sum_{\abs{\alpha} \leq L} \abs{c_\alpha}.
$$

\begin{lemma}\label{lem:estimates on polynomials} For any $d, L \in \Nb$ there exists $A=A(d,L)> 1$ such that: if $P : \Rb^d \rightarrow \Rb$ is a polynomial of degree at most $L$ with $P(0)=0$, then

\begin{enumerate} 
\item For all $0 \leq r \leq R$,  
$$
\frac{1}{A}\left( \frac{r}{R} \right)^L\max_{\norm{x} \leq R} \abs{P(x)}  \leq \max_{\norm{x} \leq r} \abs{P(x)} \leq A \frac{r}{R} \max_{\norm{x} \leq R} \abs{P(x)}.
$$
\item For all $0 \leq r \leq 1$, 
$$
\frac{1}{A} r^L\norm{P} \leq \max_{\norm{x} \leq r} \abs{P(x)} \leq A r \norm{P}.
$$
\end{enumerate}
\end{lemma} 

\begin{proof} Since norms on finite dimensional vector spaces are equivalent, there exists $A_0 = A_0(d,L) > 0$ such that: if $P : \Rb^d \rightarrow \Rb$ is a polynomial of degree at most $L$, then 
$$
\frac{1}{A_0} \max_{\norm{x} \leq 1} \abs{P(x)} \leq \norm{P} \leq A_0 \max_{\norm{x} \leq 1}\abs{P(x)}.
$$

Now if $0 \leq r \leq R$ and $P=\sum_{\abs{\alpha} \leq L} c_\alpha x^{\alpha}$ is a polynomial of degree at most $L$ with $P(0)=0$, then 
\begin{align*}
\max_{\norm{x} \leq r}& \abs{P(x)} = \max_{\norm{x} \leq 1}\abs{P(rx)} \leq A_0 \sum_{\alpha} \abs{c_\alpha} r^{\abs{\alpha}} \leq A_0 \left( \frac{r}{R}\right)\sum_{\alpha} \abs{c_\alpha} R^{\abs{\alpha}} \\
& \leq A_0^2  \left( \frac{r}{R}\right)\max_{\norm{x} \leq 1}\abs{P(Rx)}=A_0^2  \left( \frac{r}{R}\right)\max_{\norm{x} \leq R}\abs{P(x)}
\end{align*}
and likewise
\begin{align*}
\max_{\norm{x} \leq r}& \abs{P(x)} = \max_{\norm{x} \leq 1}\abs{P(rx)} \geq \frac{1}{A_0} \sum_{\alpha} \abs{c_\alpha} r^{\abs{\alpha}} \geq  \frac{1}{A_0}  \left( \frac{r}{R}\right)^L \sum_{\alpha} \abs{c_\alpha} R^{\abs{\alpha}} \\
& \geq  \frac{1}{A_0^2}  \left( \frac{r}{R}\right)^L\max_{\norm{x} \leq 1}\abs{P(Rx)}=\frac{1}{A_0^2}  \left( \frac{r}{R}\right)^L\max_{\norm{x} \leq R}\abs{P(x)}.
\end{align*}
Thus (1) is true. 

For (2), notice that when $0 \leq r \leq 1$ we have 
$$
\max_{\norm{x} \leq r} \abs{P(x)} = \max_{\norm{x} \leq 1}\abs{P(rx)} \leq A_0 \sum_{\alpha} \abs{c_\alpha} r^{\abs{\alpha}} \leq A_0 r \norm{P}
$$
and 
\begin{equation*}
\max_{\norm{x} \leq r} \abs{P(x)}= \max_{\norm{x} \leq 1}\abs{P(rx)} \geq \frac{1}{A_0} \sum_{\alpha} \abs{c_\alpha} r^{\abs{\alpha}} \geq \frac{1}{A_0} r^L \norm{P}. \qedhere
\end{equation*}

\end{proof}

\subsection{Finite order contact implies Gromov hyperbolicity}\label{sec:finite type implies GH} Finally we prove that finite order contact implies Gromov hyperbolicity. Suppose 
$$
L : = L_k(\Omega)
$$
is finite.

Fix $p \in \Omega$. Then for $x \in \partial\Omega$ and $t \in [0,1]$, let 
$$
x_t : = (1-t)x+tp.
$$
Using Theorem~\ref{thm:2 implies 1 in main theorem} it suffices to prove the following.

\begin{proposition} There exists $C > 0$ such that: if $0 <  s < t \leq 1$, $x \in \partial \Omega$, $V \in \Gr_k(\Kb^d)$, and $(x+V) \cap \Omega = \emptyset$, then 
$$
 \dist_{\rm Euc}(x_s, (x_s+V) \cap \partial \Omega) \leq C \left( \frac{s}{t} \right)^{1/L}  \dist_{\rm Euc}(x_t, (x_t+V) \cap \partial \Omega).
$$
\end{proposition} 

Suppose not. Then for every $n \geq 1$ there exist $0 < s_n < t_n \leq 1$, $x_n \in \partial \Omega$, and $V_n \in \Gr_k(\Kb^d)$ such that $(x_n+V_n) \cap \Omega = \emptyset$ and 
$$
 \dist_{\rm Euc}(x_{s_n}, (x_{s_n}+V_n) \cap \partial \Omega) \geq n \left( \frac{s_n}{t_n} \right)^{1/L}  \dist_{\rm Euc}(x_{t_n}, (x_{t_n}+V_n) \cap \partial \Omega).
$$

\begin{lemma}\label{lem:sn converges to zero} $\lim_{n \rightarrow \infty} s_n = 0=\lim_{n \rightarrow \infty} \dist_{\rm Euc}(x_{s_n}, (x_{s_n}+V_n) \cap \partial \Omega)$. \end{lemma} 

\begin{proof} Since $L_k(\Omega)<+\infty$, it suffices to show that $\lim_{n \rightarrow \infty} s_n =0$. Suppose not. Then after passing to a subsequence we can suppose that  $\inf \{s_n\} > 0$. Since $\Omega$ is bounded, 
$$
\sup \{ \dist_{\rm Euc}(x_{s_n}, (x_{s_n}+V_n) \cap \partial \Omega)\} < +\infty.
$$
Since $\inf t_n \geq \inf s_n> 0$, we must have 
$$
\inf \left\{\dist_{\rm Euc}(x_{t_n}, (x_{t_n}+V_n) \cap \partial \Omega) \right\}> 0. 
$$
Thus 
$$
+\infty = \sup\{n\} \leq \sup \left\{ \left( \frac{t_n}{s_n} \right)^{1/L}  \frac{ \dist_{\rm Euc}(x_{s_n}, (x_{s_n}+V_n) \cap \partial \Omega)}{\dist_{\rm Euc}(x_{t_n}, (x_{t_n}+V_n) \cap \partial \Omega) } \right\} <+\infty,
$$
which is a contradiction. 
   \end{proof} 

Let $d_0 : =(\dim_{\Rb} \Kb)k$. For each $n \geq 1$ fix a $\Rb$-affine isomorphism 
$$
T_n : \Rb^{d_0+1} \rightarrow \left[ x_n + \Rb( p-x_n) \oplus V_n\right]
$$
such that $T_n(0)=x_n$, $T_n(e_1) = p$, and 
$$
T_n( {\rm Span}_{\Rb}\{e_2,\dots, e_{  d_0+1}\}) = x_n+V_n.
$$

Since $p \in \Omega$ and  $(x_n + V_n) \cap \Omega =\emptyset$, the angle between $V_n$ and $\Rb(p-x_n)$ is uniformly bounded from below. Hence, we can also assume that $\{T_n\}$ is a unifomly biLipschitz sequence of maps. 

Let $\Omega_n : = T_n^{-1}(\Omega) \subset \Rb^{d_0+1}$ and $V_0 :={\rm Span}_{\Rb}\{e_2,\dots, e_{  d_0+1}\}$. Since $T(e_1) = p \in \Omega$ and $T(V_0) = x_n+V_n$ does not intersect $\Omega$, 
$$
\Omega_n \subset (0,\infty) \times \Rb^{d_0}. 
$$
Further, since $\{T_n\}$ is unifomly biLipschitz, there exists $C_0 > 0$ such that 
\begin{equation}\label{eqn:sntnOmega_n}
 \dist_{\rm Euc}(s_n e_1, (s_n e_1+V_0) \cap \partial \Omega_n) \geq C_0 n \left( \frac{s_n}{t_n} \right)^{1/L}  \dist_{\rm Euc}(t_n e_1, (t_n e_1+V_0) \cap \partial \Omega_n)
\end{equation}
for all $n \geq 1$. 

Next we can find $a,b \in (0,1)$ and for each $n$ a smooth function $f_n : (-2b, 2b)^{d_0} \rightarrow [0, a)$ such that 
$$
\Omega_n \cap (-a, a)\times (-b,b)^{d_0}  = \left\{ (t, y) \in  (-a, a)\times (-b,b)^{d_0} : t > f_n(y) \right\}. 
$$
Let $f_n = P_n + E_n$ be the order $L$ Taylor series decomposition of $f_n$. Since $\{T_n\}$ is uniformly biLipschitz,  there exists a constant $C_1 > 0$ such that 
\begin{equation}\label{eqn:error term}
\norm{E_n(y)} \leq C_1 \norm{y}^{L+1}
\end{equation}
for all $n \geq 1$ and all $y \in [-b, b]^{d_0}$.

Since $L=L_k(\Omega)$ and $\{T_n\}$ is uniformly biLipschitz, we have 
$$
\inf_{n \geq 1} \norm{P_n} > 0
$$
(recall that $\norm{P_n}$ was defined in the previous subsection). Thus by Lemma~\ref{lem:estimates on polynomials} there exists $A > 1$ such that:
\begin{itemize} 
\item For all $0 \leq r \leq R$ and $n \geq 1$,  
\begin{equation}\label{eqn:comparing maxs on different balls}
\frac{1}{A}\left( \frac{r}{R} \right)^L\max_{\norm{y} \leq R} \abs{P_n(y)}  \leq \max_{\norm{x} \leq r} \abs{P_n(y)} \leq A \frac{r}{R} \max_{\norm{y} \leq R} \abs{P_n(y)}.
\end{equation}
\item For all $0 \leq r \leq 1$ and $n \geq 1$, 
\begin{equation}\label{eqn:bounds on max of Pn}
\frac{1}{A} r^L \leq \max_{\norm{y} \leq r} \abs{P_n(y)} \leq A r. 
\end{equation} 
\end{itemize}

\begin{lemma}\label{lem:distance to boundary versus max P} There exist $C_2, T > 0$ such that: if $0 < t \leq T$ and $n \geq 1$, then
$$
\frac{1}{C_2} \dist_{\rm Euc}(t e_1, (t e_1+V_0) \cap \Omega_n) \leq \min\left\{ r : \max_{\norm{y} \leq r} \abs{P_n(y)} \geq t\right\} \leq C_2 \dist_{\rm Euc}(t e_1, (t e_1+V_0) \cap \Omega_n).
$$
\end{lemma} 

\begin{proof} Since $L_k(\Omega) < +\infty$ and $\{T_n\}$ is unifomly biLipschitz, we have  
$$
\lim_{t \searrow 0} \,  \sup_{n \geq 1} \, \dist_{\rm Euc}(t e_1, (t e_1+V_0) \cap \Omega_n) = 0.
$$
Thus we can fix $T > 0$ such that: if $n \geq 1$ and $0 < t \leq T$, then 
\begin{align}
C_12^{L+1}A^{\frac{(L+1)^2}{L}}t^{\frac{1}{L}} & \leq 1/2,\label{eqn:crazy product less than one} \\
2A^{\frac{L+1}{L}} t^{\frac{1}{L}} & \leq b<1, \text{ and} \label{eqn:t smaller than a} \\
\dist_{\rm Euc}(t e_1, (t e_1+V_0) \cap \Omega_n)&  \leq b<1. \label{eqn:t small enough so that graph(f) realizes distance} 
\end{align}

Now fix $0 < t \leq T$ and $n \geq 1$. Let 
$$
r_0 : = \min\left\{ r : \max_{\norm{y} \leq r} \abs{P_n(y)} \geq t\right\} \quad \text{and} \quad r_1 : = 2Ar_0. 
$$
Then by Equations~\eqref{eqn:bounds on max of Pn} and~\eqref{eqn:t smaller than a}
\begin{equation}\label{eqn:upper bound on r0}
r_0 < r_1 \leq 2A^{\frac{L+1}{L}} t^{\frac{1}{L}} < b < 1. 
\end{equation} 
Further,  
\begin{align*}
\max_{\norm{y} \leq r_1} f_n(y) & \geq  \max_{\norm{y} \leq r_1} \abs{P_n(y)} - C_1r_1^{L+1}&  \text{(by Equation~\eqref{eqn:error term})}\\
& \geq \frac{1}{A} \frac{r_1}{r_0}  \max_{\norm{y} \leq r_0} \abs{P_n(y)}- C_1r_1^{L+1}&  \text{(by Equation~\eqref{eqn:comparing maxs on different balls})} \\
& \geq  2t -C_12^{L+1}A^{\frac{(L+1)^2}{L}} t^{\frac{L+1}{L}} \geq t &  \text{(by Equation~\eqref{eqn:crazy product less than one})}. 
\end{align*} 
Hence 
\begin{equation}\label{eqn:distance to boundary leq max of P}
\dist_{\rm Euc}(t e_1, (t e_1+V_0) \cap \Omega_n) \leq r_1 = 2A \min\left\{ r : \max_{\norm{x} \leq r} \abs{P_n(x)} \geq t\right\}. 
\end{equation}

By Equation~\eqref{eqn:t small enough so that graph(f) realizes distance}, there exists $y_0 \in (-b,b)^{d_0}$ such that $f(y_0) = t$ and 
$$
\norm{y_0} = \dist_{\rm Euc}(t e_1, (t e_1+V_0) \cap \Omega_n).
$$
By Equations~\eqref{eqn:distance to boundary leq max of P} and~\eqref{eqn:upper bound on r0},
$$
\norm{y_0} \leq r_1 \leq 2A^{\frac{L+1}{L}} t^{\frac{1}{L}} < 1. 
$$
Then by Equations~\eqref{eqn:error term} and~\eqref{eqn:crazy product less than one}
$$
\max_{\norm{y} \leq \norm{y_0}} \abs{P_n(y)} \geq f_n(y_0) - \abs{E_n(y_0)} \geq t - C_1 \norm{y_0}^{L+1} \geq t - C_1 2^{L+1}A^{\frac{(L+1)^2}{L}} t^{\frac{L+1}{L}} \geq \frac{1}{2}t.  
$$
Then by Equation~\eqref{eqn:comparing maxs on different balls}
$$
\max_{\norm{y} \leq 2A\norm{y_0}} \abs{P_n(y)} \geq \frac{1}{A}\frac{ 2A\norm{y_0}}{ \norm{y_0}} \max_{\norm{y} \leq \norm{y_0}} \abs{P_n(y)}\geq t.
$$
Thus 
\begin{equation*}
 \min\left\{ r : \max_{\norm{y} \leq r} \abs{P_n(y)} \geq t\right\} \leq 2A\norm{y_0} = 2A \dist_{\rm Euc}(t e_1, (t e_1+V_0) \cap \Omega_n). \qedhere
 \end{equation*}

\end{proof} 

\begin{lemma} $\lim_{n \rightarrow \infty} t_n =0= \lim_{n \rightarrow \infty} \dist_{\rm Euc}(t_n e_1, (t_n e_1+V_0) \cap \partial \Omega_n)$. \end{lemma} 

\begin{proof} It suffices to show that $\lim_{n \rightarrow \infty} \dist_{\rm Euc}(t_n e_1, (t_n e_1+V_0) \cap \partial \Omega_n) =0$. Since $s_n \rightarrow 0$ (see Lemma~\ref{lem:sn converges to zero}),  Equation~\eqref{eqn:bounds on max of Pn} implies that
$$
\max_{\norm{x} \leq A^{1/L}s_n^{1/L}} \abs{P_n(x)} \geq s_n
$$ 
when $n$ is sufficiently large. Then Lemma~\ref{lem:distance to boundary versus max P} implies that 
$$
 \dist_{\rm Euc}(s_n e_1, (s_n e_1+V_0) \cap \partial \Omega_n) \leq C_2 A^{1/L} s_n^{1/L}
$$
when $n$ is sufficiently large. Since $t_n \leq 1$, Equation~\eqref{eqn:sntnOmega_n} then implies that 
$$
C_2 A^{1/L} \geq C_0 n  \dist_{\rm Euc}(t_n e_1, (t_n e_1+V_0) \cap \partial \Omega_n) 
$$
when $n$ is sufficiently large. Thus  $\lim_{n \rightarrow \infty} \dist_{\rm Euc}(t_n e_1, (t_n e_1+V_0) \cap \partial \Omega_n) =0$.

\end{proof} 

We obtain a contradiction with Equation~\eqref{eqn:sntnOmega_n} by proving the following. 

\begin{lemma} For $n$ sufficiently large, 
$$
 \dist_{\rm Euc}(s_n e_1, (s_n e_1+V_0) \cap \partial \Omega_n) \leq C_2^2 A \left( \frac{s_n}{t_n} \right)^{1/L}  \dist_{\rm Euc}(t_n e_1, (t_n e_1+V_0) \cap \partial \Omega_n).
$$
\end{lemma} 

\begin{proof} Fix $n$ sufficiently large so that $ t_n \leq T$. Let
$$
R:=\min\left\{ r : \max_{\norm{y} \leq r} \abs{P_n(y)} \geq t_n\right\} \leq C_2 \dist_{\rm Euc}(t_n e_1, (t_n e_1+V_0) \cap \partial \Omega_n).
$$

\noindent \textbf{Case 1:} Assume $\frac{As_n}{t_n} \leq 1$. Let $r : = \left(\frac{As_n}{t_n}\right)^{1/L} R$. Then Equation~\eqref{eqn:comparing maxs on different balls} implies that 
$$
\max_{\norm{y} \leq r} \abs{P_n(y)} \geq \frac{1}{A} \left( \frac{r}{R}\right)^{L} \max_{\norm{y} \leq R} \abs{P(y)}=s_n.
$$
Hence Lemma~\ref{lem:distance to boundary versus max P} implies that
\begin{align*}
 \dist_{\rm Euc} & (s_n e_1, (s_n e_1+V_0) \cap \partial \Omega_n) \leq C_2 r= C_2 \left(\frac{As_n}{t_n}\right)^{1/L} R \\
 &  < C_2^2 A \left(\frac{s_n}{t_n}\right)^{1/L}  \dist_{\rm Euc}(t_n e_1, (t_n e_1+V_0) \cap \partial \Omega_n)
\end{align*}
(the last line follows since $A > 1$ and $L > 2$). 

\medskip

\noindent \textbf{Case 2:} Assume $\frac{As_n}{t_n} > 1$. Let $r : = \frac{As_n}{t_n} \cdot R$. Then Equation~\eqref{eqn:comparing maxs on different balls} implies that 
$$
\max_{\norm{y} \leq r} \abs{P_n(y)} \geq \frac{1}{A} \frac{r}{R} \max_{\norm{y} \leq R} \abs{P(y)}=s_n.
$$
Hence Lemma~\ref{lem:distance to boundary versus max P} implies that
\begin{align*}
 \dist_{\rm Euc} & (s_n e_1, (s_n e_1+V_0) \cap \partial \Omega_n) \leq C_2 r= C_2  \frac{As_n}{t_n} \cdot R \\
 &  \leq C_2^2 A \left(\frac{s_n}{t_n}\right)^{1/L}  \dist_{\rm Euc}(t_n e_1, (t_n e_1+V_0) \cap \partial \Omega_n). \qedhere
\end{align*}

\end{proof}


\section{The minimal metric}\label{sec:minimal metric}


In this section we recall the definition of the minimal pseudo-metric from~\cite{FK2021} and then, using results from~\cite{Fiacchi2023,DDF2021},  we observe that the minimal metric is biLipschitz to $\mathfrak{q}_\Omega^{(2)}$. 

Let $\Db \subset \Rb^2$ denote the unit disk. Recall that a $\Cc^2$-smooth map $f : \Db \rightarrow \Rb^d$ is
\begin{enumerate} 
\item \emph{conformal} if for every $p \in \Db$ the derivative $df_p$ preserves angles, 
\item \emph{harmonic} if every component $f_k$ of $f$ is a harmonic function. 
\end{enumerate} 
Then, given a domain $\Omega \subset \Rb^d$ with $d \geq 3$, let ${\rm CH}(\Db, \Omega)$ denote the set of harmonic conformal maps $\Db \rightarrow \Omega$. The \emph{minimal pseudo-metric} on $\Omega$ is then defined by
$$
\mathfrak{m}_\Omega(p;v) : = \inf\{ \norm{\xi} : \exists f \in {\rm CH}(\Db, \Omega) \text{ and $\xi \in \Rb^2 \simeq T_0\Db$ with } f(0) = p, \, df_0(\xi) = v\} 
$$
where $p \in \Omega$ and $v \in \Rb^d \simeq T_p \Omega$.

\begin{proposition}\label{prop:minimal_metruc_is_Lip_to_affine} If $d \geq 3$ and $\Omega \subset \Rb^d$ is a convex domain, then 
$$
\frac{1}{2} \mathfrak{q}^{(2)}_\Omega(p;v) \leq \mathfrak{m}_\Omega(p;v) \leq  \mathfrak{q}^{(2)}_\Omega(p;v)
$$
for all $p \in \Omega$ and $v \in \Rb^d\simeq T_p \Omega$.
\end{proposition} 

\begin{proof} 
The upper bound has already been observed in~\cite[Lemma 5.5]{Fiacchi2023} and is a consequence of the fact that an isometric affine embedding of a disc is always conformal and harmonic. For the lower bound we use the following calculation from~\cite{Fiacchi2023} (see also~\cite[Section 5]{DDF2021}). 

\begin{lemma}[{\cite[Lemma 5.3]{Fiacchi2023}}] If $\mathcal{H} : = (0,\infty) \times \Rb^{d-1}$, then 
$$
\mathfrak{m}_{\mathcal{H}}(p;v) = \frac{\abs{v_1}}{2\abs{p_1}}
$$
for all $p=(p_1, \dots, p_d) \in \mathcal{H}$ and $v=(v_1,\dots, v_d) \in \Rb^d$. 
\end{lemma}

Fix $p \in \Omega$ and non-zero $v \in \Rb^d \simeq T_p \Omega$. Suppose that $\varphi : \Db \rightarrow \Omega$ is harmonic conformal, $\varphi(0) = p$, and $\varphi'(0)\xi = v$ for some $\xi \in \Rb^2 \simeq T_0\Db$. Let $V : = \varphi'(0)\Rb^2 \in \Gr_2(\Rb^d)$ and 
$$
\delta: = \dist_{\rm Euc}(p, (p+V) \cap \partial \Omega).
$$
If $\delta = \infty$, then 
$$
\frac{1}{2}\mathfrak{q}^{(2)}_\Omega(p;v)=0 \leq \mathfrak{m}_\Omega(p;v). 
$$
So assume that $\delta <\infty$. Fix a unit vector $u \in \Rb^2$ such that 
$$
x:=p+\frac{\delta}{\norm{\varphi'(0)}}\varphi'(0) u \in \partial \Omega.
$$
Since $\Omega$ is convex, there exists a real affine hyperplane $H$ such that $\Omega \cap H = \emptyset$ and $x \in H$. Let $\Omega'$ be the connected component of $\Rb^d \setminus H$ containing $\Omega$. Then by the lemma
$$
\norm{\xi} = \norm{\norm{\xi} u} \geq \mathfrak{m}_\Omega\big(p; \varphi'(0)(\norm{\xi}u)\big) \geq \mathfrak{m}_{\Omega'}\big(p; \varphi'(0)(\norm{\xi}u)\big) = \frac{\norm{\varphi'(0)(\norm{\xi}u)}}{\delta} .
$$
Since $\varphi$ is conformal, $\norm{\varphi'(0)(\norm{\xi}u)}=\norm{v}$. Hence 
$$
\norm{\xi} \geq \frac{\norm{v}}{2 \delta} \geq \frac{\norm{v}}{2\delta^{(2)}_\Omega(p;v)}=\frac{1}{2} \mathfrak{q}^{(2)}_\Omega(p;v).
$$

Since $\varphi : \Db \rightarrow \Omega$ was an arbitrary conformal map with $\varphi(0) = p$ and $v \in \varphi'(0)\Rb^2$, we then obtain that 
\begin{equation*}
\mathfrak{m}_\Omega(p;v) \geq \frac{1}{2} \mathfrak{q}^{(2)}_\Omega(p;v). \qedhere
\end{equation*}
\end{proof} 

As a corollary to Proposition~\ref{prop:minimal_metruc_is_Lip_to_affine} and Theorem~\ref{thm:smooth case} we obtain the following. 

\begin{corollary}\label{cor:characterization in smooth case for minimal metric in paper} Suppose $d \geq 3$ and $\Omega \subset \Rb^d$ is a smoothly bounded convex domain. Then the minimal metric on $\Omega$ is Gromov hyperbolic if and only if every real affine $2$-plane has finite order contact with $\partial \Omega$. 
\end{corollary}

\bibliographystyle{alpha}
\bibliography{complex}

\end{document}